\documentclass[11pt]{amsart}

\usepackage{amssymb,latexsym,amsmath,amsfonts}
\usepackage[mathscr]{eucal}
\usepackage{mathrsfs}
\usepackage{multicol}
\usepackage{graphicx}
\usepackage{caption}
\usepackage[abs]{overpic}
\usepackage{amscd}
\usepackage{amsthm}
\usepackage{amsxtra}
\usepackage[nointegrals]{wasysym}
\usepackage{bm}
\usepackage{calc}
\usepackage{float}
\usepackage[usenames,dvipsnames]{xcolor}
\usepackage{tikz-cd}
\usepackage{hyperref}
\hypersetup{colorlinks=true, citecolor=MidnightBlue, linkcolor=BrickRed}

\numberwithin{equation}{section}
\theoremstyle{definition}
\newtheorem{definition}{Definition}[section]
\theoremstyle{definition}

\newtheorem{remark}[definition]{Remark}

\theoremstyle{plain}
\newtheorem{theorem}[definition]{Theorem}
\newtheorem{lemma}[definition]{Lemma}
\newtheorem{cor}[definition]{Corollary}
\newtheorem{Prop}[definition]{Proposition}

\newcommand{\beas}{\begin{eqnarray*}}
\newcommand{\eeas}{\end{eqnarray*}}
\newcommand{\bes} {\begin{equation*}}
\newcommand{\ees} {\end{equation*}}
\newcommand{\be} {\begin{equation}}
\newcommand{\ee} {\end{equation}}
\newcommand{\bea} {\begin{eqnarray}}
\newcommand{\eea} {\end{eqnarray}}

\newcommand{\eps}{\varepsilon}
\newcommand{\zt}{\zeta}
\newcommand{\de} {\delta}

\newcommand\partl[2]{\dfrac{\partial{#1}}{\partial{#2}}}

\newcommand{\bdy}{\partial}

\newcommand{\D}{\mathbb{D}}
\newcommand{\cont}{\mathcal{C}}

\newcommand{\rea}{\operatorname{Re}}
\newcommand{\ima}{\operatorname{Im}} 

\newcommand{\A}{\mathcal{A}}
\newcommand{\B}{\mathcal B}

\newcommand{\K}{\mathcal{K}}
\newcommand{\M}{\mathcal{M}}
\newcommand{\N}{\mathcal{N}}
\newcommand{\R}{\mathcal{R}}

\newcommand{\T}{\mathbf T}

\newcommand\wt[1]{\widetilde{#1}}
\newcommand\floor[1]{\lfloor {#1} \rfloor}
\newcommand{\zobar}{\overline{z_1}}
\newcommand{\zbar}{\overline{z}}
\newcommand{\clD}{\overline\D}
\newcommand{\spa}{\operatorname{span}}
\newcommand{\sli}{\operatorname {slice}}

\newcommand{\CC}{\mathbb{C}^2}
\newcommand{\Cn}{\mathbb{C}^n}
\newcommand{\C} {\mathbb{C}} 
\newcommand{\rl}{\mathbb{R}}

\newcommand{\Z} {\mathbb{Z}}
\newcommand{\gl}{\operatorname{GL}}

\begin{document}
\title[Polynomially convex embeddings]{Polynomially convex embeddings of odd-dimensional closed manifolds}
\author{Purvi Gupta}
\address{Department of Mathematics, Indian Institute of Science, Bangalore 560012, India}
\email{purvigupta@iisc.ac.in}
\author{Rasul Shafikov}
\address{Department of Mathematics, University of Western Ontario,  London, Ontario N6A 5B7, Canada}
\email{shafikov@uwo.ca}
\thanks{P. Gupta is supported in part by a UGC CAS-II grant (Grant No. F.510/25/CAS-II/
2018(SAP-I))}
\begin{abstract} 
It is shown that any smooth closed orientable manifold of dimension $2k + 1$, $k \geq 2$, admits a smooth polynomially convex embedding into $\mathbb C^{3k}$. This improves by $1$ the previously known lower bound of  $3k+1$ on the possible ambient complex dimension for such embeddings (which is sharp when $k=1$). It is further shown that the embeddings produced have the property that all continuous functions on the image can be uniformly approximated by holomorphic polynomials. Lastly, the same technique is modified to construct embeddings whose images have nontrivial hulls containing no nontrivial analytic disks. The distinguishing feature of this dimensional setting is the appearance of nonisolated CR-singularities, which cannot be tackled using only local analytic methods (as done in earlier results of this kind), and a topological approach is required.
\end{abstract}
\maketitle


\section{Introduction}
The problem of finding the least Euclidean dimension into which all abstract manifolds of a fixed dimension admit embeddings with certain prescribed properties appears in many different contexts in geometry, as exemplified by the classical results of Nash, Grauert--Morrey, Remmert--Bishop--Narasimhan, etc. In a similar spirit, the topological consequences of imposing certain convexity-type conditions on manifolds in $\Cn$ has been a recurrent topic of interest in complex analysis; for instance, see \cite{Al93, Fo94a, ElCi15, NeSi16}. Along these lines, we study {\em the polynomially convex embedding problem}: what is the least $n$ such that all closed smooth real $m$-manifolds admit polynomially convex smooth embeddings into $\Cn$? 

\subsection*{Main result} We prove the following result in this paper. 

\begin{theorem}\label{t.main} 
Let $M$ be a closed orientable smooth manifold of real dimension $2k+1$, where $k\geq 2$. Then there is a smooth embedding $\iota:M\hookrightarrow\C^{3k}$ such that 
	\begin{enumerate}
		\item  $M'=\iota(M)$ is totally real except along a finite union of simple closed real curves in $M'$, 
		\item $M'$ is a polynomially convex subset of $\C^{3k}$, and
		\item any continuous function on $M'$ can be uniformly approximated  by holomorphic polynomials on $M'$.
	\end{enumerate}
\end{theorem}

A compact set $K\subset\mathbb C^n$ is polynomially convex if its polynomially convex hull, defined as 
$$\widehat K = \{ z\in \mathbb C^n: |P(z)|\leq \sup\nolimits_K|P|  \text{ for all holomorphic polynomials } P\ \text{on}\ \Cn\},$$
coincides with $K$. Embeddability of manifolds as polynomially convex compacts in $\mathbb C^n$ is important in view of the 
Oka--Weil theorem: if $K=\widehat K$, then any function holomorphic on some neighborhood of $K$ can be uniformly approximated on $K$ by holomorphic polynomials. A partial converse also holds: if all continuous functions on $K$ can be uniformly approximated on $K$ by holomorphic polynomials, then 
$K = \widehat K$. Thus, $(3)$ implies $(2)$ in Theorem~\ref{t.main}. However, the polynomial embeddability of $M$ is the crux of the matter, and poses the main technical challenge. 

Until recently, the best known bound on the optimal embedding dimension for polynomially convex embeddings was the same as that for {\em totally real embeddings}, i.e., where the image admits no complex lines in any of its tangent planes. These two types of embeddings are related by the result that any $m$-dimensonal totally real submanifold in $\Cn$, $m<n$, can be made polynomially convex after a small perturbation; see~\cite{FoRo93}, \cite{Fo94} and \cite{LoWo09}. It is known that if $n\geq \floor{\frac{3m}{2}}$, then any closed $m$-dimensional manifold can be embedded into $\mathbb C^n$ as a totally real submanifold. This bound is sharp for totally real embeddability (see \cite{HoJaLa12}), that is, if $n<\floor{\frac{3m}{2}}$ then one cannot always avoid {\it CR-singularities}, i.e., points where the tangent plane contains nontrivial complex subspaces. When $m\leq 3$, this bound is also sharp for polynomially-convex embeddability as no real $n$-dimensional submanifold in $\Cn$ can be polynomially convex, see \cite[\S 2.3]{St07}. However, when $m\geq 4$, the extent to which the bound for polynomially convex embeddability can be improved is not known. 

In our earlier paper \cite{GuSh20}, this bound was improved by one for even-dimensional manifolds.  That is, for $k\geq 2$, $(2k)$-dimensional manifolds admit polynomially convex smooth embeddings into $\C^{3k-1}$. Here, generic CR-singularities are isolated points. As in the case of totally real embeddings, the method of small perturbations works in this setting, with some additional local analysis required near the singularities. Theorem~\ref{t.main} extends this improvement of bound to all odd-dimensional manifolds of dimension at least $5$, under the assumption of orientability. In this setting, CR-singularities generically form closed real curves on the manifold. While isolated CR-singularities are well-studied in all relevant dimensions, starting from the seminal work~\cite{Bi65}, the literature on the global properties of CR-singular sets of positive dimension is rather sparse; see \cite{Do95} and \cite{We69} for some global results. Such sets carry nontrivial topology whose properties are not very well understood. Thus, to obtain embeddings with prescribed convexity properties, one can no longer rely on local analysis alone and must use topological methods. 

\subsection*{Idea of the proof} We briefly describe the construction that yields Part $(2)$ of Theorem~\ref{t.main}. First, by a standard argument, we observe that a generic embedding of $M$ is totally real except along a finite union of disjoint simple closed curves of CR-singularities. A tubular neighborhood in the embedded manifold of such a curve is an instance of a {\em tube enclosing a CR-singular curve in $\C^{3k}$} (see Definition~\ref{def_tubes}). We call two such disjoint tubes {\em totally real cylindrically (TRC) cobordant} if their boundaries can be joined by a totally real cylindrical manifold within $\C^{3k}$. Using a relative $h$-principle, we show that the TRC cobordism class of a (parametrized) tube is determined by the homotopy class of its {\em frame map} --- a map from $S^1\times S^{2k-1}$ into the complex Stiefel manifold $W_{2k+1,3k}$ induced by a field of frames on the given tube, see \eqref{eq_Fhat}. From this, we deduce that the possible homotopy classes of all such maps can be enumerated by $\Z$ when $k$ is even, and by $\Z\oplus \Z_2$ when $k$ is odd. We then construct models of tubes (enclosing CR-singular curves) that realize all the possible enumerations, and satisfy the crucial property that they are polynomially convex. When $k$ is even, the construction is a straightforward modification of the Beloshapka--Coffman normal form discussed in Subsection~\ref{subsec_BCform}. However, when $k$ is odd, this construction has to be adjusted using a Hopf fibration to account for the torsion that appears in this setting. To summarize, a small neighborhood of any CR-singular curve of a generic embedding of $M$ is TRC cobordant to one of the constructed polynomially convex models. Effectively, this gives the following procedure: for each CR-singular curve in the (embedded) manifold, we cut out such a tubular neighbourhood, and glue instead a TRC cobordant polynomially convex tube using a cylindrical manifold that joins their boundaries. Finally, using the fact that the CR-singular set of the new embedding admits a polynomially convex neighborhood in the manifold, we perform a small perturbation (using \cite{ArWo17}) to make the manifold globally polynomially convex.

Part (3) of Theorem~\ref{t.main} now follows from a combination of three classical approximation results, including the aforementioned result due to Oka--Weil.  This part of the theorem can be paraphrased as follows: for any closed $(2k+1)$-manifold $M$ as in Theorem~\ref{t.main}, there exist $3k$ smooth functions on $M$ that generate $C(M)$, the algebra of continuous complex-valued functions on $M$.

\subsection*{Scope for further improvements} There are several questions that remain open in the context of Theorem~\ref{t.main}. First, the assumption of orientability may simply be an artefact of our proof. In the absence of orientability, one will have to account for ``nonorientable tubes" enclosing CR-singular curves, but the broader technique still holds promise. Second, while there are examples of closed $m$-dimensional real manifolds, $m\geq 2$, that cannot be embedded into $\mathbb C^n$ when $n=\floor{\frac{3m}{2}}-1$, see \cite{HoJaLa12}, all the known examples are nonorientable when $m=2k+1$ and $k$ is odd, and it is not clear whether $n=\floor{\frac{3m}{2}}$ is sharp for orientable manifolds in this case. We note that totally real embeddability would give an easier proof of $(2)$ and a stronger version of $(3)$ in Theorem~\ref{t.main}. Finally, in the case of isolated CR-singularities, any embedding of a compact manifold {\em with boundary} can be perturbed to be totally real and polynomially convex (see~\cite{GuSh20}). It is not clear whether our proof can be modified to obtain totally real polynomially convex embeddings of $(2k+1)$-dimensional manifolds with boundary in $\C^{3k}$. 

Whether the polynomially convex embedding dimension can be improved further than $\floor{\frac{3m}{2}}-1$, for $m\geq 6$, remains an open
problem. If one only seeks topological embeddings, the sharp bound is known from \cite{VoZa71}, wherein polynomially convex topological embeddings of all $m$-dimensional smooth manifolds (in fact, even simplicial polytopes) into $\mathbb C^{m+1}$ have been constructed. However, these embeddings are highly nonsmooth. For smooth embeddings, the next possible dimensional improvement already poses several technical difficulties. One has to reckon with surfaces of CR-singularities. This causes some difficulty in both enumerating all the topological possiblities (in the TRC cobordism sense) and constructing polynomially convex models. Further, while the Beloshapka--Coffman normal form is still available, the nondenegeracy required to invoke this form cannot be guaranteed everywhere as degenerate CR-singular points generically form a set of codimension two in the manifold, and can no longer be separated from the CR-singular set by a simple transversality argument. 

\subsection*{Embeddings with no analytic disks in their hulls} The technology developed to prove Theorem~\ref{t.main} can be used to lower the bound for another embedding problem (raised in \cite{IzSt18}): {\em what is the least $n'$ such that every compact $m$-dimensional smooth manifold can be smoothly
embedded into $\C^{n'}$ as some $\Sigma$ with $\widehat\Sigma\setminus\Sigma$ nonempty but
containing no analytic disk, i.e., there is no nonconstant holomorphic map from the unit disk into $\widehat\Sigma\setminus\Sigma$?} This can also be asked for rational hulls where, if $K\subset\Cn$ is a compact set, its {\em rational hull} is 	\bes
		h_r(K)=\{z\in\Cn:p(z)\in p(K)\ \text{for all holomorphic polynomials}\ P\ {on}\ \Cn\}.
	\ees
Part of the motivation for this problem comes from the fact that all the classical constructions of nontrivial hulls with no analytic disks involve highly nonsmooth sets. For smooth embeddings, the best known bound for surfaces is $n'\leq 3$, obtained in \cite{IzSt18} via explicit embeddings. For higher dimensions, $n'\leq \floor{\frac{3m}{2}}-1$ when $m$ is even, as seen in \cite{GuSh20}, and $n'\leq\floor{\frac{3m}{2}}$ when $m$ is odd, as proved in \cite{ArWo17}. For orientable manifolds, we improve the latter as follows.

\begin{theorem}\label{thm_anstructure} Given any closed orientable smooth manifold $M$ of real dimension $2k+1$, $k\geq 2$, there is a smooth embedding of $M$ into $\C^{3k}$ with image $\Sigma$ so that $\widehat{\Sigma}\setminus \Sigma$ is nonempty but contains no analytic disk, and $\widehat \Sigma=h_r(\Sigma)$.	
\end{theorem}

\subsection*{Structure of the paper} In Section~\ref{sec_CRprelim}, we collect some basic facts about CR-singularities that allow us to make certain simplifying assumptions in our proofs. The Beloshapka--Coffman normal form, in particular, is discussed in Subsection~\ref{subsec_BCform}. In Section~\ref{sec_ATprelim}, we lay the groundwork for the topological aspects of the proofs. After establishing some essential notation in Subsection~\ref{subsec_notn}, we explicitly compute generators of certain homotopy groups associated with the Stiefel manifold $W_{2k+1,3k}$ in Subsections~\ref{subsec_homStief} and~\ref{subsec_homStief2}. Section~\ref{sec_tubes} is devoted to the study of tubular neighborhoods of certain curves in $(2k+1)$-dimensional submanifolds of $\C^{3k}$. For such neighborhoods, we introduce a topological invariant (index) in Subsection~\ref{subsec_index}, and an equivalence relation (totally real cylindrical cobordism) in Subsection~\ref{subsec_trcob}. We construct polynomially convex representatives of all possible equivalence classes in Section~\ref{subsec_models}. Finally, the proofs of the main results are carried out in Section~\ref{sec_proofs}.
\section{CR-geometric preliminaries}\label{sec_CRprelim}

We recall some facts about CR-singularities of real $m$-dimensional submanifolds in $\Cn$. Putting together these facts, we obtain Part $(1)$ of Theorem~\ref{t.main}. 

 \subsection{CR-singularities} The {\it CR-dimension} of a real submanifold 
$M\subset \mathbb C^n$ at a point $p\in M$ is the (complex) dimension of the maximal complex linear subspace $T^c_p M$ of $T_pM$, the tangent plane of $M$ at $p$ (considered as a subset of $\mathbb C^n$). If $\dim M \le n$ and $M$ is in general position, then it will be generically totally real, i.e., the CR-dimension of $M$ will be $0$ almost everywhere. In this case, a point $p\in M$ is called 
{\it CR-singular} if $\dim_\C T^{c}_pM\geq 1$. Furthermore, we call $p \in M$ a {\em CR-singularity of order $\mu$} if $\dim_{\C} T^{c}_p M = \mu$. Given $\mu\in\mathbb N_0$, if
$S_\mu$ denotes the set of CR-singular points of $M$ of order $\mu$, then $S_\mu$ is either empty 
or is a (not necessarily closed) submanifold of $M$ of dimension $m - (2\mu^2 + 2\mu (n-m))$ and  
\begin{equation}\label{e.smu}
\overline{S_\mu} = \bigcup_{\nu\ge \mu} S_\nu,  
\end{equation}
see ~\cite{Do95} for more details. From this it follows that if $n \ge \lfloor 3m/2 \rfloor$, then any $m$-dimensional manifold admits a totally real embedding into 
$\mathbb C^n$. If $m = 2k$ and $n=3k-1$, then all the CR-singularities of a generic $M$ are isolated points, while if 
$m= 2k+1$ and $n=3k$, then a generic $M$ has a one-dimensional set of CR-singularities of order $1$ 
and no CR-singularities of order $2$ or higher. Thus, we have

\begin{lemma}\label{lem_emb1} Let $\iota:M\hookrightarrow\C^{3k}$ be a generic smooth embedding of a smooth $(2k+1)$-dimensional closed manifold. Then the set of CR-singularities of $\iota(M)$ is a finite union of smooth simple closed real curves.
\end{lemma}

\subsection{The Beloshapka--Coffman normal form}\label{subsec_BCform} Building on Beloshapka's work for $n=5$, see \cite{Be97}, Coffman introduced in \cite{Co97} a notion of {\em nondegeneracy} for a CR-singular point of an $m$-dimensional manifold $M$ in $\Cn$, where $\frac{2}{3}(n+1)\leq m<n$, $n\geq 5$. He showed that near a nondegenerate CR-singular point, $M$ is locally {\em formally} equivalent to the {\em Beloshapka--Coffman normal form}:
	\bea\label{eq_BCform}
\B^{m,n}=
\left\{(z_1,...,z_n)\in\C^{n}:
\begin{aligned}
	&  \ima z_j=0,\ 2\leq j\leq m-1,\\
	& z_{m}=\zobar^2,\\
	&z_{m+1}=|z_1|^2+\zobar(\rea z_2+i\rea z_3),\\
	&z_\ell
		=\zobar (\rea z_{2(\ell-m)}+i\rea z_{2(\ell-m)+1}),\ m+2\leq \ell\leq n\\
\end{aligned}
\right\}.
	\eea
Note that $\B^{m,n}$ is totally real except along a $(3m-2n-2)$-dimensional plane, where it has CR-singularities of order $1$. In \cite{Co10}, Coffman further showed that if $M$ is real-analytic near a nondegenerate CR-singularity, then there is a normalizing transformation that converges, i.e., $M$ is locally biholomorphically equivalent to $\B^{m,n}$ near such a point. We note that the nondegeneracy conditions required for the formal equivalence to hold at a CR-singular point $p\in M$ are full-rank conditions on matrices involving the second-order derivatives of the local graphing functions of $M$ at $p$; see equations $(58)$ and $(62)$ in \cite[\S 6.]{Co10}. These conditions will generically yield a codimension $2$ set in $M$. When $m=2k+1$ and $n=3k$, $k\geq 2$, the set of CR-singularities is of dimension $1$. Thus, by transversality, the CR-singular set of a generic $(2k+1)$-dimensional $M\subset\C^{3k}$ consists only of nondegenerate points. Combining this with the density of real-analytic functions (in any fixed $\cont^\ell$-norm, $0\leq \ell<\infty$), and Lemma~\ref{lem_emb1}, we obtain the following crucial preparatory result.

\begin{lemma}\label{lem_emb2}
Any smooth embedding $\iota : M\hookrightarrow \mathbb C^{3k}$ of a smooth $(2k+1)$-dimensional closed manifold admits a small perturbation (in any fixed $\cont^s$-norm, $1\leq s<\infty$) that gives a smooth embedding $\iota_0:M\hookrightarrow \mathbb C^{3k}$ such that 
	\begin{enumerate}
\item [$(i)$] the set of CR-singularities of $\iota_0(M)$ is a finite union of smooth simple closed real curves, and 
\item [$(ii)$] near any CR-singularity, $\iota_0(M)$ is locally biholomorphic to the Beloshapka--Coffman normal form $\B^{2k+1,3k}$.
\end{enumerate}
\end{lemma}

We note that $\B^{2k+1,3k}$ is locally polynomially convex at the origin (see Lemma~\ref{lem_polcvx}), which is a property that is invariant under local biholomorphisms. Thus, since a submanifold in $\Cn$ is always locally polynomially convex at its totally real points, $\iota_0(M)$ in Lemma~\ref{lem_emb2} is locally polynomially convex everywhere. 

\section{Topological Preliminaries}\label{sec_ATprelim}
We collect some algebro-topological results that will play a crucial role in our arguments. Note that the homotopy groups computed in this section are known in the literature, and will be familiar to topologists. However, since our interest lies in finding explicit generators of these groups, we carry out certain computations from scratch.

\subsection{Notation}\label{subsec_notn} We fix some notation for the rest of this paper. Slightly different coordinates are used for points in even and odd-dimensional spheres, as follows. 
\begin{itemize}
\item [$(a)$] For $k\geq  2$, 
	\beas 
		&S^{2k}&=\left\{(z,w,\zt,s)\in\C\times\C\times\C^{k-2}\times\rl:
|z|^2+|w|^2+||\zt||^2+s^2=1\right\},\\
		&S^{2k-1}&=
		\left\{(z,t,\zt,s)\in\C\times\rl\times\C^{k-2}\times\rl:
		|z|^2+t^2+||\zt||^2+s^2=1\right\},\\
		&D^{2k}&=
			\left\{(z,t,\zt,s)\in\C\times\rl\times\C^{k-2}\times\rl:
|z|^2+t^2+||\zt||^2+s^2\leq 1\right\}.
	\eeas

\smallskip
\item [$(b)$] In $(a)$ above, when needed, we write $z=x+iy$, $w=s+it$ and $\zt=(u_1+iv_1,...,u_{k-2}+iv_{k-2})$.

\smallskip
\item [$(c)$] When convenient, we denote $e^{i\theta}\in S^1\subset\C$ by $\theta$, where $\theta\in(-\pi,\pi]$.    
\end{itemize}  

We denote the set of all orthonormal $k$-frames in $\Cn$ by $W_{k,n}$. An element $A=[a_1,...,a_k]\in W_{k,n}$ will be represented as an $n\times k$ matrix with orthonormal columns $a_j=(a_{1j},...,a_{nj})^T\in\Cn$, $j=1,...,k$. When $k=1$, $W_{1,n}=S^{2n-1}$, and we switch between the two conventions, with the understanding that 
\beas
	S^{2n-1}\ni(z,t,\zt,s)\leftrightarrow
		\begin{pmatrix} z\\ t+is\\ \zt^T
		\end{pmatrix}\in W_{1,n}.
\eeas 
Let $i_r:W_{k,n}\mapsto W_{k+r,n+r}$ be the map given by 
\bea\label{eq_stiefjump}
	 A=
	\begin{pmatrix}
	a_{11} & \cdots & a_{1k}\\
	\vdots & \ddots & \vdots \\
	a_{n1} & \cdots  & a_{nk}
	\end{pmatrix}
				\mapsto 
\left(\begin{array}{c | c}	
		A & \mathbf{0}_{n,r}\\
	\hline
  	\mathbf{0}_{r,k} & \mathbf{I}_{r,r}
 \end{array}\right).
\eea
Note that $i_{r_1}\circ i_{r_2}=i_{r_1+r_2}$. We also need to consider $V_{k,n}$, the (noncompact) Stiefel manifold of $\C$-linearly independent $k$-frames in $\C^{n}$. The compact space $W_{k,n}$ is a strong deformation retract of $V_{k,n}$ via Gram-Schmidt orthogonalization. Lastly, $e_j$ denotes the vector $(0,\cdots,\underbrace{1}_{z_j},\cdots,0)\in\Cn$.   

\subsection{Higher homotopy groups of complex Stiefel manifolds}\label{subsec_homStief} To make the exposition as self-contained as possible, we first collect some basic definitions and results. For more details, the reader may consult classical references such as \cite{Ha02} and \cite{Ad74}.  

A continuous map $p:E\rightarrow B$ between topological spaces is said to be a {\em Hurewicz fibration} if it has the homotopy lifting property with respect to all topological spaces, i.e., for any space $Z$, map $g:Z\rightarrow E$ and homotopy $H:Z\times[0,1]\rightarrow B$ such that the following diagram commutes, there is a $G:Z\times [0,1]\rightarrow E$ that extends $g$ and lifts $H$.
\beas
\begin{tikzcd}
  Z\times\{0\} \arrow[r, "g"] \arrow[d, hook] & E \arrow[d, "p"] \\
 Z\times [0,1] \arrow[r, "H"] \arrow [ur, dashed, "G"] & B
\end{tikzcd}
\eeas
In this case, the {\em fibers} or inverse images of points in $B$ are homotopy equivalent (to $F$, say), and the fibration is denoted by 
	\bes
		F\hookrightarrow E \xrightarrow{p} B.
	\ees
Fixing some $b\in B$, $e\in p^{-1}(b)$ and $f=e$, the fibration induces the following long exact sequence of homotopy groups.
	\bes
	\cdots \rightarrow \pi_{\ell+1}(B,b)\xrightarrow{\de_\ell} \pi_\ell(F,f)
	\xrightarrow{\iota^*} \pi_\ell(E,e)\xrightarrow{p^*} \pi_\ell(B,b)
		\xrightarrow{\de_{\ell-1}} \pi_{\ell-1}(F,f)\rightarrow \cdots
	\ees
Here, if $H:(D^{\ell+1},\bdy D^{\ell+1})\rightarrow (B,b)$ represents an element in $\pi_{\ell+1}(B,b)$, then for any $G:D^{\ell+1}\rightarrow E$ such that the diagram 
\bea\label{diag_bdymap}
\begin{tikzcd}
  D^\ell\times\{0\} \arrow[r, "g\: \equiv\: e"] \arrow[d, "\iota"] 
		& E \arrow[d, "p"] \\
 D^\ell\times[0,1]\cong D^{\ell+1} \arrow[r, "H"] \arrow [ur, dashed, "G"] 
		& B
\end{tikzcd}
\eea
commutes, $G|_{\bdy D^{\ell+1}}$ induces a well-defined element in $\pi_\ell(F,f)$. This is because $G(\bdy (D^{\ell}\times[0,1]))\subseteq p^{-1}(b)\cong F$ and $G:D^\ell\times\{0\}\mapsto e\ (=f$ in $F$). The map $\de_\ell$ maps $H$ to this element in $\pi_\ell(F,f)$.
%
%


Next, we recall the Freudenthal suspension theorem which, in particular, implies the stability of the homotopy groups $\pi_{m+1}(S^m)$, $m\geq 3$. Given a topological space $X$, its suspension is the space 
\bes
\Sigma X=\{(x,t)\in X\times[0,1]:(x_1,t)\sim (x_2,t)\ \text{when either}\ t=0\ \text{or}\ t=1 \}.
\ees
Given a map $f:X\rightarrow Y$, its suspension is the map $\Sigma f:\Sigma X\rightarrow\Sigma Y$ given by $(x,s)\mapsto (f(x),s)$. The Freudenthal suspension theorem says that if $X$ is an $n$-connected CW complex, then $f\mapsto\Sigma f$ induces an isomorphism between $\pi_i(X)$ and $\pi_{i+1}(\Sigma X)$ when $i<2n+1$, and an epimorphism between $\pi_{2n+1}(X)$ and $\pi_{2n+2}(\Sigma X)$. Applying this iteratively to spheres, we obtain that
	\beas
		&\pi_{2k}(S^{2k-1})\cong \pi_4(S^3)\cong\Z_2,&\quad k\geq 2,\\
		&\Z\cong \pi_3(S^2)\twoheadrightarrow \pi_{4}(S^3).& 
	\eeas
Furthermore, since the Hopf fibration given by $\mathfrak h:(z,w)\mapsto (2\zbar w,|z|^2-|w|^2)$ is a generator of $\pi_3(S^2)$, we obtain the following generator of $\pi_{2k}(S^{2k-1})$ via suspensions:
	\be\label{eq_gen}
		\mathfrak h_k:
		(z,w,\zt,s)
	\mapsto \left(2\zbar w,|z|^2-|w|^2,\zt\sqrt{1+|z|^2+|w|^2},s\sqrt{1+|z|^2+|w|^2}\right). 
	\ee

In \cite{Gi67b} (results announced in \cite{Gi67}), fibrations and suspensions are used to compute certain higher homotopy groups of complex Stiefel manifolds. In particular, it is shown that
	\be\label{eq_homStfl}
		\pi_{\ell}(W_{k,n})=
		\begin{cases}
		0,&\ \text{if}\ \ell\leq 2(n-k),\\
		\Z,&\ \text{if}\ \ell=2(n-k)+1,\\
		0,&\ \text{if}\ \ell=2(n-k)+2,\ n-k\ \text{is odd},\\
		\Z_2,&\ \text{if}\ \ell=2(n-k)+2,\ n-k\ \text{is even}.		
		\end{cases}
	\ee
We retrace this technique to compute explicit generators of $\pi_{2k-1}(W_{2k+1,3k})$ and $\pi_{2k}(W_{2k+1,3k})$, $k\geq 2$. 

\begin{lemma}\label{lem_gentors} For any $k\geq 2$, the map 
	\beas
	\alpha:S^{2k-1}\ni(z,t,\zt,s)\mapsto i_{2k}(z,t,\zt,s)
		=\left(\begin{array}{c | c}
		\begin{array}{c}
			z\\ t+is\\ \zeta^T
		\end{array}
		& \mathbf{0}_{k,2k} \\
		\hline
		  		\mathbf{0}_{2k,1} & \mathbf{I}_{2k,2k}
 \end{array}\right)
	\eeas
represents a generator of $\pi_{2k-1}(W_{2k+1,3k})\cong \Z$. 

When $k>2$ is odd, the map 
	\beas
	\beta:S^{2k}\ni(z,w,\zt,s)\mapsto 
\left(\begin{array}{c | c}
		\begin{array}{c}
			2\zbar w\\ 1-2|w|^2+i2|w|s\\ 
			2|w|\zeta^T
		\end{array}
		& \mathbf{0}_{k,2k} \\
		\hline		
		  \mathbf{0}_{2k,1} & \mathbf{I}_{2k,2k}
 \end{array}\right)
	\eeas
represents a generator of $\pi_{2k}(W_{2k+1,3k})\cong \Z_2$.
\end{lemma}
\begin{proof} To prove the lemma, we will show that the generators of $\pi_\ell(W_{2k+1,3k})$ essentially descend to generators of $\pi_\ell(S^{2k-1})$, $\ell=2k-1,2k$.
 
We break the proof into three steps ((A), (B) and (C) below). First, for $1\leq k\leq n$, we consider the fibration 
	\be\label{eq_fibration}
		W_{k-1,n-1}\xrightarrow{i_1} W_{k,n}\xrightarrow{p} S^{2n-1},
	\ee
where $p:[v_1,...,v_k]\mapsto v_k$. Ignoring the basepoints, we obtain the following long exact sequence
	\bes
	\cdots \rightarrow \pi_{\ell+1}(S^{2n-1})\xrightarrow{\de_\ell} \pi_\ell(W_{k-1,n-1})\xrightarrow{i_1^*}
	\pi_\ell(W_{k,n})\xrightarrow{p^*} \pi_\ell(S^{2n-1})\rightarrow \cdots
	\ees
which, applied iteratively, proves that for $\ell<2n-2m$, $\pi_\ell(W_{k-m,n-m})\cong\pi_\ell(W_{k,n})$ via $i_m^*$. In particular,
	\begin{center}
(A)\hspace{20pt} for $k\geq 1$ , $\pi_{2k}(W_{2k+1,3k})\cong\pi_{2k}(W_{2,k+1})$ via $i_{2k-1}^*$.
	\end{center}
and, $\pi_{2k-1}(S^{2k-1})\cong\pi_{2k-1}(W_{2k+1,3k})$ via the isomorphism $i_{2k}^*$. Note that the latter fact proves the first half of our claim. 

Next, we consider a particular case of \eqref{eq_fibration}: 
	$S^{2k-1}\cong W_{1,k}\xrightarrow{i_1} W_{2,k+1}\xrightarrow{p}
	S^{2k+1}$, to obtain the long exact sequence 
	 \be\label{eq_les}
	\cdots \rightarrow \pi_{2k+1}(S^{2k+1})\xrightarrow{\de_{2k}}
	\pi_{2k}(S^{2k-1})\xrightarrow{i_1^*} \pi_{2k}(W_{2,k+1})
	\xrightarrow{p^*} \underbrace{\pi_{2k}(S^{2k+1})}_{=0}\rightarrow\cdots.
	\ee
In order to understand the map $\de_{2k}$, we fix $b=e_{k+1}\in \C^{k+1}$ and $e=[e_k,e_{k+1}]\in W_{2,k+1}$, and consider the following specific case of diagram \eqref{diag_bdymap} 
\beas
\begin{tikzcd}
  D^{2k}\cong D^{2k}\times \{0\} \arrow[r, "g\: \equiv\: e"] \arrow[d, "\iota"] 
	& W_{2,k+1} \arrow[d, "p"] \\
 D^{2k+1}\cong D^{2k}\times [0,1] \arrow[r, "H"] \arrow [ur, dashed, "G"] 	
	& S^{2k+1}\subset \C^{k+1}
\end{tikzcd}
\eeas
for any $H=(h_1,...,h_{k+1}):(D^{2k+1},\bdy D^{2k+1})\rightarrow (S^{2k+1},b)$. When $k+1$ is even, the above diagram commutes if $G:(D^{2k+1},\bdy D^{2k+1})\rightarrow (W_{2,k+1},e)$ is set as
	\bes
	(D^{2k+1})\ni Z	\mapsto 
	\begin{pmatrix}
	h_2(Z) & h_1(Z) \\
	-h_1(Z) & h_2(Z) \\
	\vdots & \vdots\\
	h_{k+1}(Z) & h_{k}(Z)\\
	-h_k(Z) & h_{k+1}(Z)
	\end{pmatrix}.
	\ees
In particular, $G|_{\bdy D^{2k+1}}\equiv [e_k,e_{k+1}]\subset p^{-1}(e_{k+1})\cong S^{2k-1}$, which induces the trivial element in $\pi_{2k}(S^{2k-1})$. Thus,
\begin{center}
	(B)\hspace{20pt} when $k$ is odd, $\de_{2k}=0$ in \eqref{eq_les} and, hence, 
		$\pi_{2k}(S^{2k-1})\cong \pi_{2k}(W_{2k+1,3k})$ via $i_1^*$.
\end{center} 
Combining $(A)$ and $(B)$, we have that $\pi_{2k}(W_{2k+1,3k})\cong \pi_{2k}(S^{2k-1})\cong\Z_2$ via $i^*_{2k}$. Moreover, owing to \eqref{eq_gen}, $\psi=i_{2k}\circ \mathfrak h_k$ is a generator of $\pi_{2k}(W_{2k+1,3k})$. To complete our proof, we show that 
	\begin{center}
	(C)\hspace{20pt} 		$\psi$ and $\beta$ are homotopic as maps from $S^{2k}$ into $W_{2k+1,3k}$. 
	\end{center}
Note that we may write $\beta=i_{2k}\circ h$, where $h:S^{2k}\rightarrow S^{2k-1}$ is given by 
	\bes
	h:(z,w,\zt,s)
		\mapsto 
			(2\zbar w,1-2|w|^2,2|w|\zeta,2|w|s).
	\ees
Then, the map $H:S^{2k}\times[0,1]\rightarrow W_{2k+1,3k}$ which sends $(z,w,\zt,s,\tau)$ to
	\beas
		i_{2k}\left(2\zbar w,|z|^2-|w|^2+\tau||\zt||^2+\tau s^2,
(\zt,s)\sqrt{2(1-\tau)|z|^2+2(1+\tau)|w|^2+(1-\tau^2)(||\zt||^2+s^2)}\right)
	\eeas
is a homotopy between $\psi$ and $\beta$. This completes the proof of our lemma. 
\end{proof}

\begin{remark}\label{rmk_deg} Viewing $S^{2k-1}$ as a sphere in $\C^k$ with complex orientation and coordinates $(z,t+is,\zt)$, we fix the convention that $[\alpha]=1$ in $\Z$. Now, for any continuous map $f:S^{2k-1}\mapsto \pi_{2k-1}(W_{2k+1,3k})\cong\Z$, the homotopy class of $f$ is a well-defined integer, which we denote by $\deg(f)$ and call the {\em degree of $f$}. The proof of Lemma~\ref{lem_gentors} shows that if $g$ is a continuous self-map of $S^{2k-1}$, then $\deg(f\circ g)=\deg(f)\deg_{S^{2k-1}}(g)$, where $\deg_{S^{2k-1}}(\cdot)$ is to be understood as the degree of a continuous self-map of $S^{2k-1}$.  
 \end{remark}

\subsection{Maps from \texorpdfstring{$\bf{S^1\times S^{2k-1}}$ }\ into certain Stiefel manifolds.}\label{subsec_homStief2} In this section, we denote by $[X,Y]$ the set of homotopy classes of continuous maps from $X$ to $Y$, where $X$ and $Y$ are topological spaces.  As a consequence of the Freudenthal suspension theorem, it is known that the set $[X,Y]$ is canonically an abelian group if $Y$ is $n$ connected and $X$ is a CW complex of dimension at most $2n$; see \cite[Cor. 3.2.3]{Ko96}. This is indeed the case when $X=S^1\times S^{2k-1}$ and $Y=W_{2k+1,3k}$, $k\geq 2$. Thus, $[S^1\times S^{2k-1},W_{2k+1,3k}]$ is an abelian group. In fact, we can be more precise.   

\begin{lemma}\label{lem_homclass} Let $k\geq 2$. Then, 
	\beas
		&[S^1\times S^{2k-1},W_{2k+1,3k}]\cong \pi_{2k-1}(W_{2k+1,3k})\oplus
		\pi_{2k}(W_{2k+1,3k})\cong
		\begin{cases}
		\Z,\ \text{when $k$ is even};\\
	  \Z\oplus\Z_2,\ \text{when $k$ is odd}.
		\end{cases}
	\eeas 
In the first case, the isomorphism is induced by the map $f\mapsto  f_{\sli}:=f|_{\{\cdot\}\times S^{2k-1}}$. In the second case, $f\mapsto f_{\sli}$ determines the projection onto the first factor. 
\end{lemma}
\begin{proof}In order to describe the abelian group $[S^1\times S^{2k-1},W_{2k+1,3k}]$, we consider the following cofiber sequence (see \cite[Chapter 3]{Ar11}):
	\be\label{eq_cofiber}
		S^{2k-1}\xrightarrow{\mu}S^1\vee S^{2k-1}\xrightarrow{\iota} S^1\times S^{2k-1}
		\xrightarrow{\gamma}
		S^{2k}\xrightarrow{\Sigma\mu} S^{2}\vee S^{2k}\xrightarrow{\Sigma\iota}
		\cdots,
	\ee
where 
	\begin{itemize}
		\item $\mu$ is the attaching map of the top cell of $S^1\times S^{2k-1}$,
		\item $\iota$ is the cellular inclusion, and
		\item $\gamma$ is the quotient map $S^1\times S^{2k-1}\rightarrow S^1\times S^{2k-1}/S^1\vee S^{2k-1}\cong S^{2k}$. 
	\end{itemize}
For any CW complex $X$, \eqref{eq_cofiber} induces the following long exact sequence of pointed sets (on fixing basepoints).
	\bes
		1\rightarrow [S^{2k},X]\xrightarrow{\gamma_*} [S^1\times S^{2k-1},X]
		\xrightarrow{i_*} [S^1\vee S^{2k-1},X]\xrightarrow{\mu_*} [S^{2k-1},X].
	\ees
Here, we have used the fact that the suspension of an attaching map is nullhomotopic. Further, $\mu_*:\pi_1(X)\times\pi_{2k-1}(X)\rightarrow\pi_{2k-1}(X)$ is the Whitehead product, which vanishes if $\pi_1(X)=0$. Thus, for $X=W_{2k+1,3k}$, $k\geq 2$, we get the following exact sequence of abelian groups:
	\bes
		0\rightarrow \pi_{2k}(W_{2k+1,3k})\rightarrow [S^1\times S^{2k-1},W_{2k+1,3k}]
		\rightarrow \pi_{2k-1}(W_{2k+1,3k})\rightarrow 0.
	\ees
Based on \eqref{eq_homStfl}, we have the result.  
\end{proof}

Next, we describe the generators of these groups. In particular, when $k$ is odd, we obtain two nonhomotopic maps from $S^1\times S^{2k-1}$ into $W_{2k+1,3k}$ that restrict to homotopic maps on each slice $\{\cdot\}\times S^{2k-1}$. 

\begin{lemma}\label{lem_tori} Let $k\geq 2$. Consider the following maps from $S^1\times S^{2k-1}$ into $W_{2k+1,3k}$. 
\beas
	&f_1:(\theta,z,t,\zt,s)\:\mapsto&\!i_{2k}(z,t,\zt,s), \\
	&f_2:(\theta,z,t,\zt,s)\:\mapsto&\! i_{2k}(ze^{i\theta},u,\zt,s).
\eeas
Then, $[(f_1)_{\sli}]=[(f_2)_{\sli}]=+1$ which, for even $k$, implies that $[f_1]=[f_2]$. However, $[f_1]\neq[f_2]$, when $k$ is odd. 
\end{lemma}
\begin{proof} We note that owing to Lemma~\ref{lem_homclass} and Lemma~\ref{lem_gentors}, we only need to show that $[f_1]\neq [f_2]$ when $k$ is odd. We fix the basepoint $b=(0,1,0,...,0)\in S^{2k-1}$ and let $S^1\times_b S^{2k-1}$ denote the pinched torus $S^1\times S^{2k-1}/S^1\times\{b\}$, where we use the same coordinates as those on $S^1\times S^{2k-1}$, but denote the pinched point by $[\theta,b]$. Since $f_j(\theta,b)=i_{2k}(b)$ for all $\theta\in[0,2\pi]$, $j=1,2$, these maps factor through $S^1\times_b S^{2k-1}$ (via $f_j^*$, say). 
\beas
\begin{tikzcd}
  S^1\times S^{2k-1} \arrow[r, "f_j"] \arrow[d, two heads] 
	& W_{2k+1,3k}  \\
 S^1\times_b S^{2k-1}\arrow [ur, dashed, "f_j^*"] 	
\end{tikzcd}
\eeas
Furthermore, any homotopy $H:S^1\times S^{2k-1}\times[0,1]\rightarrow W_{2k+1,3k}$ between $f_1$ and $f_2$ descends to a homotopy between $f_1^*$ and $f_2^*$ as long as the $2$-sphere given by $H(S^1\times\{b\}\times[0,1])$ in $W_{2k+1,3k}$ is nullhomotopic. This is always the case since $\pi_2(W_{2k+1,3k})=0$. Thus, it suffices to show that $f_1^*$ and $f_2^*$ are nonhomotopic as maps from $S^1\times_b  S^{2k-1}$ into $W_{2k+1,3k}$.  

For this, we consider the maps $g_j:=f_j^*\circ g$, where $g:S^{2k}\rightarrow S^1\times_b S^{2k-1}$ is given by 
	\bes
		(z,w,\zt,s)\mapsto 
		\begin{cases}
			(\arg w, 2\overline z|w|,1-2|w|^2,2|w|\zt,2|w|s),&\ w\neq 0;\\
			[\theta,b],&\  w=0.
		\end{cases}
	\ees
Any homotopy between $f_1^*$ and $f_2^*$ will extend to a homotopy between $g_1$ and $g_2$. Since $g_2=\beta$ is a generator of $\pi_{2k}(W_{2k+1,3k})=\Z_2$ (as shown in Lemma~\ref{lem_gentors}), our proof is complete if we show that $g_1$ is nullhomotopic. For this, it suffices to show that $\wt g:S^{2k}\rightarrow S^{2k-1}$ given by
	\bes
		\wt g:(z,w,\zt,s)\mapsto (2\zbar|w|,1-2|w|^2,2|w|\zt,2|w|s)
	\ees
is nullhomotopic, for $g_2=i_{2k}\circ \wt g$. The homotopy $H:S^{2k}\times[0,1]\rightarrow S^{2k-1}$ given by
\beas
	H:\big((z,w,\zt,s),\tau\big)\mapsto h\big(\eta(\tau,|w|)z,\tau+(1-\tau)|w|,\eta(\tau,|w|)\zt,\eta(\tau,|w|)s\big),
\eeas
where $\eta(\tau,r)=\frac{\sqrt{1-(\tau+(1-\tau)r)^2}}{\sqrt{1-r^2}}$, resolves this matter; see the schematic in Figure~\ref{fig_hom}. The map $H$ is well-defined when $|w|=1$ because 
\bes
\lim_{r\rightarrow 1^-}\eta(\tau,r)=1-\tau\ \text{for all}\ \tau\in[0,1].
\ees
Moreover, $H\big((z,w,\zt,s),0\big)=h(z,|w|,\zt,s)=h(z,w,\zt,s)$ and $H\big((z,w,\zt,s),1\big)\equiv (0,-1,0,0)$. This concludes our proof.   

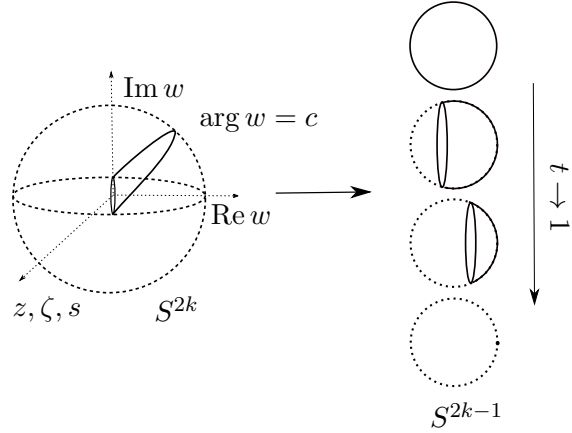
\begin{figure}[H]
\begin{overpic}[grid=false,tics=10,scale=0.11]{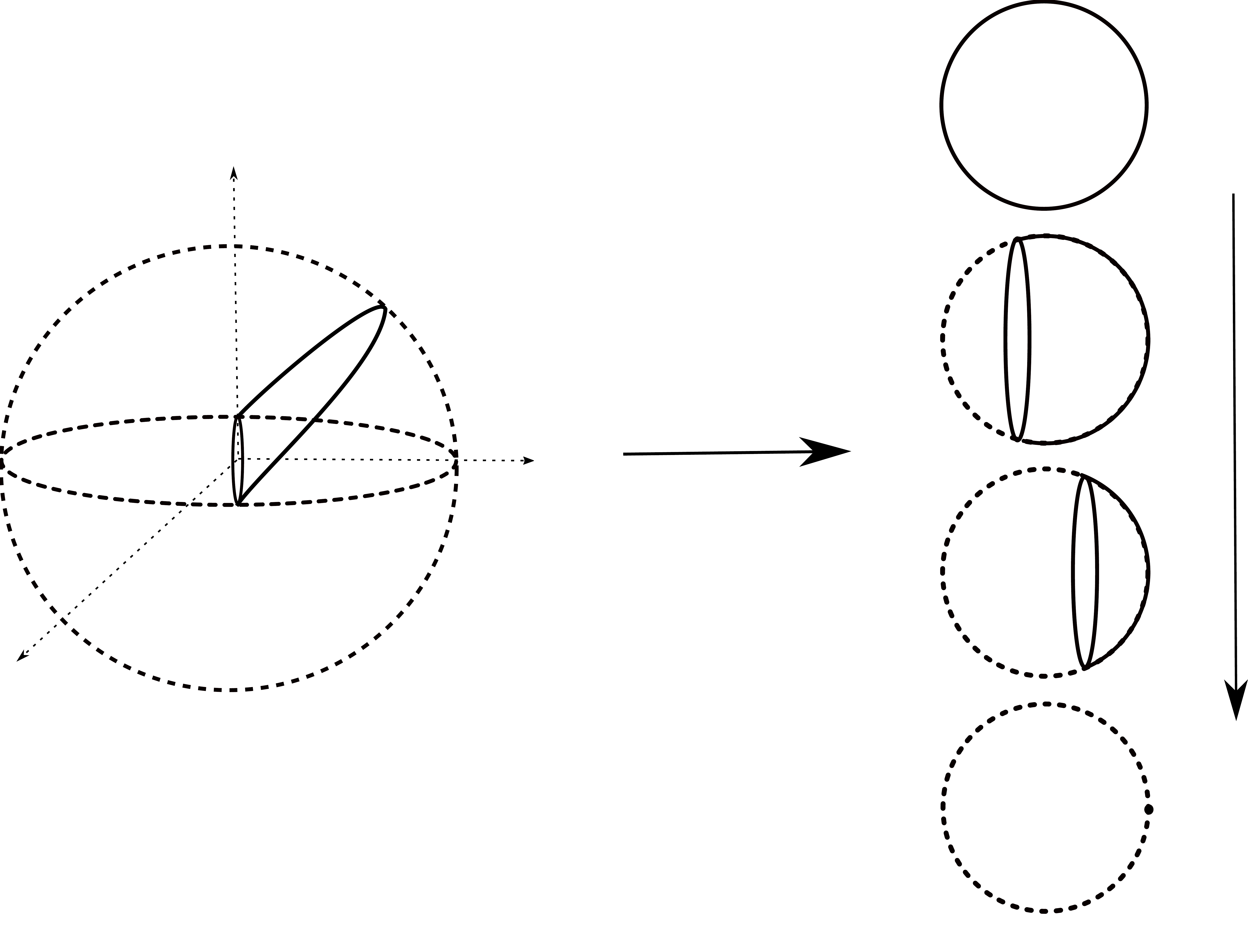} 
\put(67,105){ $\arg w=c$}
\put(50,31){ $S^{2k}$}
\put(158,-9){$S^{2k-1}$}
\put(75,68){$\rea w$}
\put(42,116){$\ima w$}
\put(0,33){$z,\zt,s$}
\put(203.5,91){\rotatebox{-90}{$t\rightarrow 1$}}
\end{overpic}
\medskip
\caption{Evolution of the homotopy $H$ along any slice $\arg w=c$.}\label{fig_hom}
\end{figure}
\end{proof}

\section{Tubular neighborhoods of curves of CR-singularities}\label{sec_tubes}

We now focus on neighborhoods of certain simple closed curves in $(2k+1)$-dimensional submanifolds of $\C^{3k}$. We introduce a specialized notion of cobordism between the boundaries of these neighborhoods, and produce polynomially convex models to represent all the possible classes of this topological equivalence. This section encapsulates the main techincal part of our construction.

\subsection{The index of a parametrized tube enclosing curve}\label{subsec_index}
The precise class of neighborhoods under consideration is as follows.  

\begin{definition}\label{def_tubes}
Suppose $\N\subset\C^{3k}$ is a compact submanifold with boundary that admits a smooth parametrization $F:S^1\times D^{2k}\rightarrow\N$ such that 
	\begin{itemize}
		\item [$a.$] $F$ is totally real on $(S^1\times D^{2k})\setminus(S^1\times\{0\})$,
		\item [$b.$] $\dim_\C T^c_p\N$ is constant over all $p\in F(S^1\times\{0\})$.
		\end{itemize}
Then we call $\N$ a {\em tube enclosing a curve in $\C^{3k}$}, or a {\em tube enclosing a curve}, since the ambient dimension is fixed. We call $F$ an {\em admissible parametrization of $\N$}. We single out two special cases.	
	\begin{enumerate}
\item [$(i)$] If $\N$ is totally real everywhere, we call it a {\em fully-TR tube}.
\item [$(ii)$] If the CR dimension of $\N$ is positive on $\gamma:=F(S^1\times\{0\})$, then $\gamma$ is independent of $F$. In this case, we say that $\N$ is a {\em tube enclosing the CR-singular curve $\gamma$}, or simply a {\em tube enclosing a CR-singular curve}.
\end{enumerate} 
In all the cases above, the word {\em parametrized} is appended to the terminology when discussing the pair $(\N,F)$. 
\end{definition} 

As before, we use $(\theta,z=x+iy,t,\zt,s)$ for points in $S^1\times D^{2k}$. 
Further, we fix the frames 
	\bes
		\vartheta=\partl{}{\theta}\ 
		\ \text{and}\ \
		\sigma=(\vartheta,\sigma')
	\ees
for the tangent bundles $T(S^1\times\{0\})$ and $T(S^1\times D^{2k})$, respecitvely, and 
\bes
	\quad 
		\sigma'= 
	\left(\partl{}{x},\partl{}{y},\partl{}{t},\partl{}{s},\partl{}{u_1},
	\partl{}{v_1},...,\partl{}{u_{k-2}},\partl{}{v_{k-2}}\right)
\ees
for the normal bundle $N(S^1\times\{0\})$ of $S^1\times\{0\}$ in $S^1\times D^{2k}$. Given a diffeomorphism $F:S^1\times D^{2k}\rightarrow\C^{3k}$ that is totally real away from the curve $F(S^1\times\{0\})$, and $p\in S^1$, we define the maps 
		\bea\label{eq_Fhat}
	&\widehat F:S^1\times S^{2k-1}\rightarrow W_{2k+1,3k}&\notag\\
	&\xi\mapsto \left[(F_*\sigma)(F(\xi))\right]^{\operatorname{GS}}
		=[DF(\xi)]^{\operatorname{GS}}, 
	\eea
and 	
	\bea\label{eq_Fhatslice}
	&\widehat F_{p\operatorname{-}\sli}:S^{2k-1}\rightarrow W_{2k+1,3k}&\notag\\
	&\eta\mapsto \left[(F_*\sigma)(F(p,\eta))\right]^{\operatorname{GS}}
		=[DF(p,\eta)]^{\operatorname{GS}},
	\eea
where $\left[(F_*\sigma)(F(\cdot))\right]^{\operatorname{GS}}$ denotes the Gram-Schmidt orthogonalization of the frame $(F_*\sigma)(F(\cdot))$, viewed as a full-rank $(3k)\times(2k+1)$ complex matrix. The map $\widehat F$ will be referred to as the {\em frame map} of the parametrized tube $(\N=F(S^1\times D^{2k}),F)$. 

Note that $\widehat F$ and $\widehat F_{p\operatorname{-}\sli}$ induce elements in $[S^1\times S^{2k-1},W_{2k+1,3k}]$ and $\pi_{2k-1}(W_{2k+1,3k})$, respectively. Since $W_{2k+1,3k}$ is a strong deformation retract of $V_{2k+1,3k}$, we will often abuse notation and drop the notation $[\cdot]^{\operatorname{GS}}$ to view $\widehat F$ and $\widehat F_{p\operatorname{-}\sli}$ as maps into $V_{2k+1,3k}$.

In this subsection, we focus on the homotopy class induced by $\widehat F_{p\operatorname{-}\sli}$. It is clear that this class is independent of $p$. Thus, when discussing the homotopy class of  $\widehat F_{p\operatorname{-}\sli}$, we assume that $p=1$ and denote $\widehat F_{p\operatorname{-}\sli}$ simply by $\widehat F_{\sli}$. Furthermore, in view of Remark~\ref{rmk_deg}, the following definition is well-defined.

\begin{definition}
The {\em index} of a parametrized tube $(\mathcal N,F)$ enclosing a curve is defined as
	\bes
		\operatorname{ind}_{F}(\mathcal N)=
			\deg(\widehat F_{p\operatorname{-}\sli}),\quad \text{for any}\ p\in S^1.
	\ees
\end{definition}

To understand the dependence of this index on $F$, we note that $\sigma'$ assigns an orientation, say $\mathfrak o$, on the normal bundle of $S^1\times\{0\}$ in $S^1\times D^{2k}$. Given a self-diffeomorphism, $\varphi$, of $S^1\times D^{2k}$ that maps $S^1\times \{0\}$ to itself, if $\varphi_*\sigma'$ induces the same orientation $\mathfrak o$ on $N(S^1\times\{0\})$, we say that $\varphi$ is {\em fiberwise orientation preserving}. Otherwise, $\varphi$ is called {\em fiberwise orientation reversing}. We now show that, modulo its sign, $\operatorname{ind}_{F}(\N)$ is independent of $F$. 

\begin{lemma}\label{lem_index}	 Let $F$ and $G$ be admissible parametrizations of $\N$, a tube enclosing a curve in $\C^{3k}$. Then, 
	\be\label{eq_invindex}
		\operatorname{ind}_{G}(\mathcal N)
			=\begin{cases}
				\operatorname{ind}_{F}(\mathcal N),&\ 
	\text{if $F^{-1}\circ G$ is fiberwise orientation-preserving};\\
				-\operatorname{ind}_{F}(\mathcal N),&\ 
	\text{if $F^{-1}\circ G$ is fiberwise orientation-reversing}.
		\end{cases}
	\ee
Moreover, if $U$ is a neighborhood of $F(A\times S^{2k})$, where $A\subset S^1$ is an arc containing $p=1$, and $H:U\rightarrow\C^{3k}$ is a biholomorphism, then 
	\bes
		\operatorname{ind}_{F}(\mathcal N)
			=\deg\left(\widehat{(H\circ F)}_{\sli}\right).
	\ees	
\end{lemma}
\begin{proof} Let $\varphi:=F^{-1}\circ G$. Since $\sigma$ is the coordinate frame on $S^1\times D^{2k}$, we may write
	\beas
		\widehat{G}_{\sli}(\eta)
			=(G_*\sigma)(G(1,\eta))=DG(1,\eta)
		= DF(\varphi(1,\eta))\cdot D\varphi(1,\eta),\qquad \eta\in S^{2k-1}.
	\eeas
Since $\det D\varphi\neq 0$ on $S^1\times D^{2k}$, the map $D\varphi(1,\cdot)|_{S^{2k-1}}$, extends to a map from $D^{2k}$ into $\operatorname{GL}_\C(2k+1)$, and is, thus, homotopic to the constant identity map in $\operatorname{GL}_\C(2k+1)$. As $V_{2k+1,3k}$ is closed under the action of $\operatorname{GL}_\C(2k+1)$ via multiplication from the right, $\widehat G_{\sli}$ is homotopic to the map 
	\bes
		\eta\mapsto DF(\varphi(1,\eta)),\qquad \eta\in S^{2k-1}.
	\ees

Next, letting $\eta^*=(\zbar,t,\zt,s)$, for $\eta=(z,t,\zt,s)\in D^{2k}$, we observe that $\varphi|_{\{1\}\times S^{2k-1}}$ is homotopic to the identity map $\mathfrak i$ on $S^{2k-1}$ when $\varphi$ is fiberwise orientation-preserving, and to $\mathfrak i^*:\eta\mapsto\eta^*$, $\eta\in S^{2k-1}$, when $\varphi$ is fiberwise orientation-reversing. Thus, $\widehat G_{\sli}$ is homotopic in $V_{2k+1,3k}$ to 
	\bes
		\begin{cases}
			DF(\mathfrak i(1,\cdot)),&\ 
	\text{if $F^{-1}\circ G$ is fiberwise orientation-preserving};\\
				DF(\mathfrak i^*(1,\cdot)),&\ 
	\text{if $F^{-1}\circ G$ is fiberwise orientation-reversing}.
		\end{cases}
	\ees
Since $\deg_{S^{2k-1}}(\mathfrak i)=+1$ and $\deg_{S^{2k-1}}(\mathfrak i^*)=-1$,  the first part of our claim now follows from Remark~\ref{rmk_deg}.  

For the second part of our claim, note that
	\bes
		 \widehat{(H\circ F)}_{\sli}(\eta)
		=D_\C H(F(1,\eta))\cdot DF(1,\eta),\quad \eta\in S^{2k-1}.
	\ees
But, $\big((D_\C H)\circ F\big)(1,\cdot)|_{S^{2k-1}}$ extends to a map from $D^{2k}$ into $\operatorname{GL}_\C(3k)$. Thus, it is homotopic to the constant identity map in $\operatorname{GL}_\C(3k)$. Now, since, $V_{2k+1,3k}$ is closed under the action of $\operatorname{GL}_\C(3k)$ via multiplication from the left, we are done. 
\end{proof}

\begin{remark}\label{rem_index} From the lemma above, it is clear that $\operatorname{ind}_{F\circ\varphi}(\mathcal N)=-\operatorname{ind}_{F}(\mathcal N)$, where $\varphi:S^1\times D^{2k}\rightarrow S^1\times D^{2k}$ is given by $(\theta,\eta)\mapsto(g(\theta),\eta^*)$, for some diffeomorphism $g:S^1\rightarrow S^1$. 
\end{remark} 

\subsection{Totally real cylindrical cobordisms between tubes enclosing curves}\label{subsec_trcob} We now discuss a notion of equivalence between parametrized tubes enclosing curves, under which one tube can be essentially replaced by another without any addition of CR-singularities. We first fix some additional notation.
\begin{itemize}
\item [$(a)$] Let $\mathbf{T}=[0,1]\times S^1\times S^{2k-1}$.
\item [$(b)$] For $\de\in(0,1/2)$, let $\mathbf T_{\de}=[0,\de]\times S^1\times S^{2k-1}$ and $\mathbf T_{1-\de}=[1-\de,1]\times S^1\times S^{2k-1}$. 
\item [$(c)$] Let $L_c:[c,c+1)\times S^1\times S^{2k-1}\rightarrow S^1\times (D^{2k}\setminus\{0\})$ be the map $(\tau,\theta,\eta)\mapsto (\theta,(c+1-\tau)\eta)$.
\end{itemize}

\begin{definition}\label{def_trc}
Suppose $(\N_1,F_1)$ and $(\N_2,F_2)$ are parametrized tubes enclosing curves in $\C^{3k}$. If they are disjoint, they are said to be {\em totally real cylindrically cobordant} or {TRC} cobordant  if there exists a totally real embedding $F:\T\rightarrow \C^{3k}$ such that, for some sufficiently small $\de$, 
	\be\label{eq_trtcob}
		F(\tau,\theta,\eta)=\begin{cases}
			(F_1\circ L_0)(\tau,\theta,\eta),
				&\ \text{on}\ \T_\de,\\
			(F_2\circ L_{1-\de})(\tau,\theta,\eta),
					&\ \text{on}\ \T_{1-\de},
	\end{cases}
	\ee
for $(\tau,\theta,\eta)\in\T$ (see Figure~\ref{fig_trc}). If they are not disjoint, then they are said to be  totally real cylindrically cobordant if there is some translation $\tau$ of $\C^{3k}$ such that $\tau(\N_2)\cap\N_1=\emptyset$ and   $(\N_1,F_2)$ and $(\tau(\N_2),F_2\circ\tau)$ are  totally real cylindrically cobordant in the sense discussed above.
\end{definition}
Before we give a characterization of TRC cobordant pairs, we state a geometric consequence of \eqref{eq_trtcob}. This result will be referred to as the {\em neighborhood replacement result} in this paper. 

\begin{Prop}\label{prop_NRR}
Suppose $(\N_1,F_1)$ and $(\N_2,F_2)$ are disjoint totally-real cobordant parametrized tubes enclosing curves in $\C^{3k}$. Then, for a sufficiently small $\eps>0$, there is a smooth embedding $\Psi_{\N_1,\:\N_2}:S^1\times 2D^{2k}\rightarrow\C^{3k}$ such that 
	\begin{enumerate}
\item $\Psi_{\N_1,\:\N_2}$ is totally real on $(S^2\times 2D^{2k})\setminus(S^1\times\{0\})$.
\item $\Psi_{\N_1,\:\N_2}(\theta,r\eta)=F_1(\theta,(r-1)\eta)$ for $r\in[2-\eps,2]$ and $(\theta,\eta)\in S^1\times S^{2k-1}$. 
\item $\Psi_{\N_1,\:\N_2}$ coincides with $F_2$ on $S^1\times D^{2k}$. 
\end{enumerate}
\end{Prop}
\begin{proof} 
Suppose $F:\T\rightarrow\C^{3k}$ is an embedding satisfying \eqref{eq_trtcob}. Let $\tau:[1-\de,2-\de]\rightarrow[0,1]$ be a smooth bijective map such that 
	\bes
		\tau(r)=\begin{cases}
			2-\de-r,& \text{when}\ r\in[1-\de,1],\\
			2-r,& \text{when}\ r\in[2-\de,2].
				\end{cases}
	\ees
Then the map $\Xi:S^1\times 2D^{2k}$ given by 
	\bes
		\Xi(\theta,r\eta)=\begin{cases}
			F_2(\theta,r\eta),\ &\text{when}\ r\in[0,1-\de],\\
			F_2(\theta,r\eta)=F(2-\de-r,\theta,\eta),
				\ &\text{when}\ r\in[1-\de,1],\\
			F(\tau(r),\theta,\eta),\ & \text{when}\ r\in [1,2-\de],\\
			F_1(\theta,(r-1)\eta)=F(2-r,\theta,\eta),\
			 & \text{when}\ r\in[2-\de,2],
		\end{cases}
	\ees
is smooth (see Figure~\ref{fig_trc2}). Furthermore, $\Xi$ is a local embedding everywhere in the interior of its domain, and is an embedding in an open neighborhood of $K=\left(S^1\times D^{2k}\right)\cup \bdy\left(S^1\times 2D^{2k}\right)$. Thus, by the relative version of the weak Whitney embedding theorem, there is an embedding $\Psi_{\N_1,\N_2}$ with the desired properties.  
\end{proof}

\begin{multicols}{2}

\begin{figure}[H]
\begin{overpic}[grid=false,tics=10,scale=0.8]{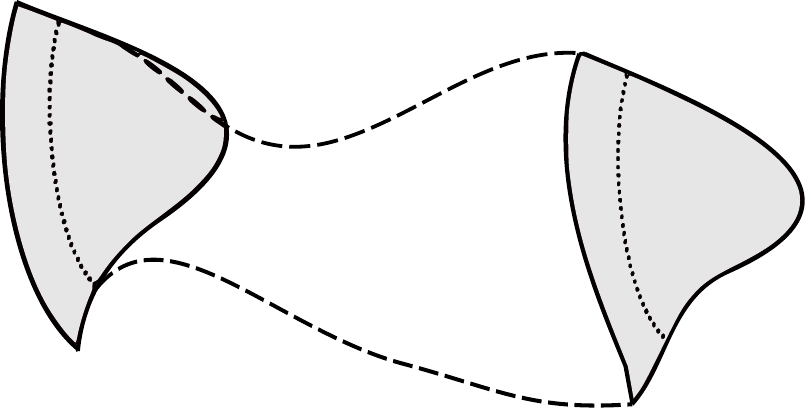}
\put(25,57){$\N_1$}
\put(150,57){$\N_2$}
\put(12,5){\tiny $\tau=0$}
\put(25,19){\tiny $\tau=\de$}
\put(135,-10){\tiny $\tau=1-\de$}
\put(160,12){\tiny $\tau=1$}
\end{overpic}
\smallskip
\caption{$\N_1\cup F(\T)\cup\N_2$}\label{fig_trc}
\end{figure}
\columnbreak    
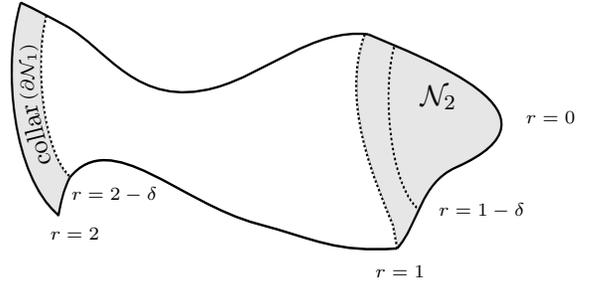
\begin{figure}[H]
\vspace{-7pt}
\begin{overpic}[grid=false,tics=10,scale=0.8]{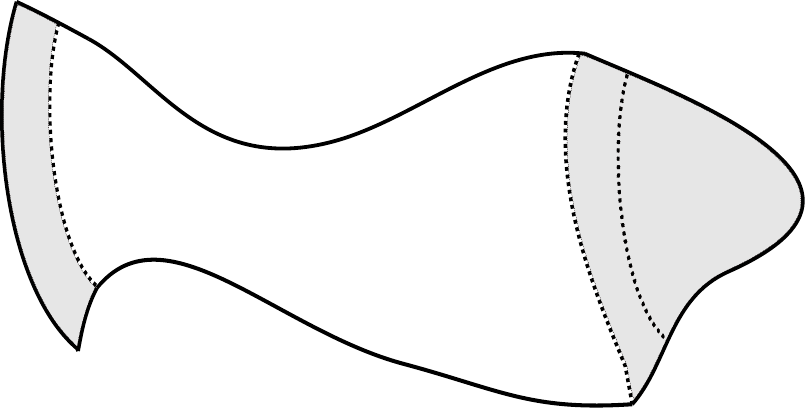}
\put(3,33){\rotatebox{104}{\small collar}}
\put(3.5,57){\rotatebox{90}{\tiny $(\bdy\N_1)$}}
\put(155,55){$\N_2$}
\put(15,4){\tiny $r=2$}
\put(23,19){\tiny $r=2-\de$}
\put(138,-10){\tiny $r=1$}
\put(162,13){\tiny $r=1-\de$}
\put(195,48){\tiny $r=0$}
\end{overpic}
\bigskip
\caption{$\Xi(S^1\times 2D^{2k})=F(\T)\cup\N_2$}\label{fig_trc2}
	\end{figure}
\end{multicols}

We now state the aforementioned algebro-topological characterization of pairs of TRC cobordant parametrized tubes enclosing curves. In particular, we see that the index is only a partial invariant in this respect. 

\begin{lemma}\label{lem_TRcob} Suppose $(\N_1,F_1)$ and $(\N_2,F_2)$ are disjoint parametrized tubes enclosing curves in $\C^{3k}$. Then, they are TRC cobordant if and only if their frame maps, $\widehat F_1$ and $\widehat F_2$, induce the same element in $[S^1\times S^{2k-1},W_{2k+1,3k}]$. 
\end{lemma}
\begin{proof} First, we fix a smooth translation-invariant (in $\tau$) field of frames on $\T$ as follows. 
	\bes
		\sigma_\T(\tau,\theta,\eta)
		=(L_{\tau}^*\sigma)(\tau,\theta,\eta).
	\ees
For any parametrization $g:S^1\times D^{2k}\rightarrow\C^{3k}$, let $g^c=g\circ L_c$. Then, because of the translation-invariance,
	\bes
		(g^c_*\sigma_\T)(g^c(c,\cdot))=(g_*\sigma)(g(\cdot))|_{S^1\times S^{2k-1}},
	\ees
and, if $\tau>c$, then $(g^c_*\sigma_\T)(g_c(\tau,\cdot))$ is homotopic to $(g_*\sigma)(g(\cdot))|_{S^1\times S^{2k-1}}$. In short, setting 
	\bes
		\widehat g_\tau^{\:c}:\xi\mapsto 	
			[(g^c_*\sigma_\T)(g^c(\tau,\xi))]^{\operatorname{GS}},
				\quad  \xi\in S^1\times S^{2k-1},
\ees
we have that for any $c\in\rl$ and $\tau\in[c,c+1)$, 
	\be\label{eq_ghat}
 \text{$\widehat g_\tau^{\:c}$ and $\widehat g$ induce the same element in $[S^1\times S^{2k-1}, W_{2k+1,3k}]$.}\tag{$\dagger$}
	\ee

Now, if $(\N_1,F_1)$ and $(\N_2,F_2)$ are TRC cobordant, then $F$ exists so that \eqref{eq_trtcob} holds on $\T$. Then, $(\tau,\theta,\eta)\mapsto (F_*\sigma_\T)(F(\tau,\theta,\eta))$, $0\leq \tau\leq 1-\de$ is a homotopy between $\widehat{F_1}$ and $\widehat{F_2}$ in $W_{2k+1,3k}$. 

To prove the converse, we rely on the relative h-principle for ample differential relations (\cite[18.4.1]{ElMi02}). For this, we first gather some important observations about the relevant spaces. Wherever possible, we follow the notation employed in \cite{ElMi02}.

\begin{enumerate}
\item Let $J^1(\T,\C^{3k})$ denote the space of $1$-jets of smooth maps from $\T$ into $\C^{3k}\cong\rl^{6k}$.  Since $\T$ is parallelizable, we may write $J^1(\T,\rl^{6k})\cong\T\times \C^{3k}\times M_{2k+1,3k}$, where $M_{2k+1,3k}$ is the space of complex $(2k+1)\times (3k)$ matrices. We use the field of frames $\sigma_\T$ to make this identification. This way, if $G:\T\rightarrow\C^{3k}$ is a smooth map, then $DG(\cdot)\equiv (G_*\sigma_{\T})(G(\cdot))$ (as matrices). Thus, $Z\mapsto (Z,G(Z),(G_*\sigma_{\T})(G(Z)))$ is a holonomic section of $J^1(\T,\C^{3k})$.
\smallskip

\item We recall that $W_{2k+1,3k}\subset V_{2k+1,3k}\subset M_{2k+1,3k}$, and if $A, B\in V_{2k+1,3k}$, then any homotopy between $[A]^{\operatorname{GS}}$ and $[B]^{\operatorname{GS}}$ in $ W_{2k+1,3k}$ lifts to a homotopy in $ V_{2k+1,3k}$.
\smallskip

\item We let $\R_{\operatorname{TR}}$ denote the differential relaton in $J^1(\T,\C^{3k})$ corresponding to totally real immersions $f:\T\rightarrow\C^{3k}$. Then, $\R_{\operatorname{TR}}$ is an open ample differential relation (see \cite[19.3.1]{ElMi02}).
\end{enumerate}

The relative $h$-principle states that if $\mathfrak H:\T\rightarrow \R_{\operatorname{TR}}\subset J^1(\T,\C^{3k})$ is a formal solution of $\R_{\operatorname{TR}}$ that is a genuine solution near $\bdy\T$, then there is a homotopy of formal solutions $\mathfrak H_\tau:\T\rightarrow \R_{\operatorname{TR}}$, $\tau\in[0,1]$, joining $\mathfrak H_0=\mathfrak H$ with a genuine solution $\mathfrak H_1$ of $\R_{\operatorname{TR}}$ such that for all $\tau$, $\mathfrak H_\tau=\mathfrak H$ near $\bdy\T$. 

To complete the proof of our claim, we note that if $(\N_1,F_1)$ and $(\N_2,F_2)$ are disjoint parametrized tubes enclosing curves in $\C^{3k}$, then there exists a smooth map $G:\T\rightarrow \C^{3k}$ such that 
	\be\label{eq_WhitEmbThm}
		G=\begin{cases}
			(F_1\circ L_0),
				&\ \text{on}\ \T_\de,\\
			(F_2\circ L_{1-\de}),
					&\ \text{on}\ \T_{1-\de}.
	\end{cases}
	\ee
for $\de\in(0,1/2)$ small enough. For this, we fix a $\de<<1/2$ and apply the isotopy version of the Weak Whitney Embedding Theorem 
to produce an isotopy between $(F_1\circ L_0)(\T\cap\{\tau=\de\})$ and $(F_2\circ L_{1-\de})(\T\cap\{\tau=1-\de\})$. This isotopy continuously connects $F_1\circ L_0$ on $\T_\de$ and $F_2\circ L_{1-\de}$ on $\T_{1-\de}$ to give a continuous map $\wt G:\T\rightarrow\C^{3k}$ that satisfies \eqref{eq_WhitEmbThm}. Now, by the relative version of the Whitney Approximation Theorem, $\wt G$ can be approximated by a smooth $G$ satisfying \eqref{eq_WhitEmbThm}. 

Now, suppose $\widehat{F_1}$ and $\widehat{F_2}$ induce the same homotopy class in $[S^1\times S^{2k-1},W_{2k+1,3k}]$. In view of $(b)$ and \eqref{eq_ghat} above, there is a homotopy  $H:[\de,1-\de]\times S^1\times S^{2k-1}\rightarrow V_{2k+1,3k}$ between $\widehat{(F_1)}^0_\de=(G_*\sigma_{\T})(G(\de,\cdot))$ and $\widehat{(F_2)}^{1-\de}_{1-\de}=(G_*\sigma_{\T})(G(1-\de,\cdot))$. Setting $\wt{H}:\T\mapsto V_{2k+1,3k}$ as 
	\bes
			\wt H(\tau,\xi)=\begin{cases}
			(G_*\sigma_\T)(G(\tau,\xi)),
				&\ \text{on}\ \T_\de,\\
			H(\tau,\xi), &\ \text{when}\ \tau\in[\de,1-\de],\\
			(G_*\sigma_\T)(G(\tau,\xi)),
					&\ \text{on}\ \T_{1-\de}.
	\end{cases}
	\ees
Then, $\mathfrak H:(\tau,Z)\mapsto (Z,G(Z),\wt H(\tau,Z))$ is a formal solution of $\R_{\operatorname{TR}}$ that is a genuine solution near $\bdy\T$. Thus, by the $h$-principle cited above, there is a genuine solution of $\R_{\operatorname{TR}}$ that coincides with $\mathfrak H$ near $\bdy\T$, which gives a totally real immersion $\wt F:\T\rightarrow\C^{3k}$ satisfying \eqref{eq_trtcob}. Now, by another application of the (relative) Whitney Embedding Theorem, $\wt F$ can be approximated by a totally real embedding satisfying \eqref{eq_trtcob}.  
\end{proof}

The above lemma shows that the TRC cobordism class of any parametrized tube $(\N,F)$ enclosing a curve in $\C^{3k}$ contains all tubes of the form $(A(\N),F\circ A)$, where $A:\C^{3k}\rightarrow\C^{3k}$ is of the form $z\mapsto rz+b$, $r>0$ and $b\in\C^{3k}$. 

\subsection{Polynomially convex models}\label{subsec_models}
So far, we have shown that there is an injective map 
\bea\label{eq_isom}
\begin{Bmatrix}
\text{equivalence classes of parametrized tubes}\\
\text{enclosing curves under TRC cobordism}
\end{Bmatrix}\hookrightarrow 
[S^1\times S^{2k-1},W_{2k+1,3k}].
\eea
In this section, we establish the surjectivity of this map. In fact, we construct polynomially convex representatives of the generators of $[S^1\times S^{2k-1},W_{2k+1,3k}]$, and then indicate how tubes of other indices can be obtainted in a similar way. Our construction is a modification of the Beloshapka--Coffman normal form \eqref{eq_BCform}. 

Consider the following set in $\C\times\rl^{2k-1}$.
	\bes
		S=\left\{(z_1,x_2,...,x_{2k})\in\C\times\rl^{2k-1}:
			|z_1|^2+(x_{3})^2+\cdots +(x_{2k-1})^2
			+\left(\sqrt{x_2^2+x_{2k}^2}-1\right)^2=\frac{1}{4} \right\}.
	\ees
Note that $S$ is a neighborhood of the unit circle in the $x_2x_{2k}$-plane. Now, consider the following two graphs over $S$ in $\C^{3k}$.  
	\bea\label{eq_modelsing}
\M^1:=
\left\{Z\in\C^{3k}:
\begin{aligned}
	& (z_1,x_2,...,x_{2k})\in S,\\
	&  y_j=0,\ 2\leq j\leq 2k,\\
	& z_{2k+1}=\zobar^2,\\
	&z_{2k+2}=|z_1|^2+\zobar\left(\sqrt{x_2^2+x_{2k}^2}-1\right)+i\zobar x_3,\\
	&z_\ell
		=\zobar (x_{2(\ell-2k-1)}+ix_{2(\ell-2k-1)+1}),\ 2k+3\leq \ell\leq 3k,\\
\end{aligned}
\right\}
	\eea
and
	\bea\label{eq_modelsing2}
\M^2:=
\left\{Z\in\C^{3k}:
\begin{aligned}
	& (z_1,x_2,...,x_{2k})\in S,\\
	&  y_j=0,\ 2\leq j\leq 2k,\\
	&z_{2k+1}=\zobar^2(x_2+ix_{2k}),\\
	&z_{2k+2}=|z_1|^2+\zobar\left(\sqrt{x_2^2+x_{2k}^2}-1\right)+i\zobar x_3,\\
	&z_\ell
		=\zobar (x_{2(\ell-2k-1)}+ix_{2(\ell-2k-1)+1}),\ 2k+3\leq \ell\leq 3k,\\
\end{aligned}
\right\}.
	\eea
Since $\M^1$ and $\M^2$ are totally real except along the circle
	\bes
		\gamma=\left\{(0,x_2,0,...,0,x_{2k},0,...,0)\in\Cn:x_2,x_{2k}\in\rl,\ 
			x_2^2+x_{2k}^2=1\right\},
	\ees
where they have CR dimension $1$, they are tubes enclosing the CR-singular curve $\gamma$. We now prove the crucial fact the $\M_1$ and $\M_2$ are polynomially convex subsets of $\C^{3k}$. As a side note, we also observe, that they are biholomorphically equivalent. 
\begin{lemma}\label{lem_polcvx}  There is a biholomorphism $\Theta$ defined in a neighborhood of $\M^1$ in $\C^{3k}$ such that $\M^2=\Theta(\M^1)$. Moreover, $\gamma$, $\M^1 $ and $\M^2 $ are polynomially convex.
\end{lemma}
\begin{proof} For the first part of the claim, consider the map 
	\bes
	\Theta:\left(z_1,...,z_{2k},z_{2k+1},z_{2k+2},...,z_{3k}\right)\mapsto 
		\left(z_1,...,z_{2k},z_{2k+1}(z_2+iz_{2k}),z_{2k+2},...,z_{3k}\right),
	\ees
which is a biholomorphism from $\C^{3k}\setminus\{z_2+iz_{2k}=0\}$ onto itself. It is clear that $\Theta(\M^1 )=\M^2 $. Since $\M^1 $ is a compact subset of $\C^{3k}\setminus\{z_2+iz_{2k}=0\}$, $\Theta$ is defined in a (sufficiently small) neighborhood of $\M^1 $ in $\C^{3k}$. 

Now, for the second claim, we first note that the polynomial convexity of $\gamma$ follows from that fact that it is a compact subset of a totally real plane (the $x_2x_{2k}$-plane). Next, we recall the following criterion (an iterated version of Theorem 1.2.16 from~\cite{St07}). {\em If $X\subset\Cn$ is a compact subset and if $F:X\rightarrow \rl^m$ is a map whose components are in $\mathcal{P}(X)$, then $X$ is polynomially convex if and only if $F^{-1}(\boldsymbol t)$ is polynomially convex for each $\boldsymbol t\in\rl^m$.} Now, choosing the restriction to $M ^j$ of the map $F:\C^{3k}\rightarrow\C^{2k-1}$ given by $Z\mapsto(z_2,...,z_{2k})$, 
and noting that the subalgebra generated by $z$ and $c\zbar^2$, $c\neq 0$, in $\cont\left(\frac{1}{2}\clD\right)$ coincides with $\cont\left(\frac{1}{2}\clD\right)$ (see~\cite{MI76}), we have that every fibre of $F$ in either  $M ^1$ or $M ^2$ is polynomially convex. Hence, the claim. 
\end{proof}

We now fix parametrizations of $\M^1$ and $\M^2$. Let $\iota:S^1\times D^{2k}\rightarrow S$ be the map given by
		\bes
		\iota:(\theta,z,t,\zt,s)\mapsto\left(\frac{1}{2}x,-\frac{1}{2} y,\left(1+\frac{1}{2} t\right)\cos\theta,\frac{1}{2} s
			,\frac{1}{2} u_1,\frac{1}{2} v_1,...,\frac{1}{2} u_{k-2},\frac{1}{2} v_{k-2},\left(1+\frac{1}{2} t\right)\sin\theta\right).
	\ees
As each $M ^j$ is a graph over $S $, the parametrization of $S $ can be pushed forward via the graphing map to obtain a parametrizing map of $M^j $, which we denote by $H^j$, $j=1,2$. 
\begin{theorem}\label{thm_inequivmodels} Let $k\geq 2$. Then, $\operatorname{ind}_{H^1}(\M^1 )=\operatorname{ind}_{H^2}(\M^2 )=1$. Moreover,  
	\begin{itemize}
	\item [$(i)$] when $k$ is even, $[\widehat H_1]=[\widehat H_2]$ in $[S^1\times S^{2k-1},W_{2k+1,3k}]$, but
%
	\item [$(ii)$] when $k$ is odd, $[\widehat H_1]\neq [\widehat H_2]$ in $[S^1\times S^{2k-1},W_{2k+1,3k}]$.
\end{itemize}

\end{theorem}
\begin{proof} As before, since $W_{2k+1,3k}$ is a deformation retract of $V_{2k+1,3k}$, we ignore the effect of $[\cdot]^{\operatorname{GS}}$ in the definitions of $\widehat H^j$ and $\widehat H^j_{\sli}$, $j=1,2$. Now, due to Lemmas~\ref{lem_homclass} and ~\ref{lem_tori}, it suffices to show that $[\widehat H^1]=[f_2]$ and $[\widehat H^2]=[f_2]$ in $[S^1\times S^{2k-1},W_{2k+1,3k}]$, where $f_1$ and $f_2$ are the maps defined in Lemma~\ref{lem_tori}. We describe a homotopy between $\widehat H^1$ and $f_1$ in some detail, and note that an almost identical procedure gives a homotopy between $\widehat H^2$ and $f_2$. 

Recall that we use the coordinates $\xi=(\theta,z=x+iy,t,\zt_1=u_1+iv_1,...,\zt_{k-2}=u_{k-2}+iv_{k-2},s)$ on $S^1\times S^{2k-1}$. Upto a cyclical permutation of the columns (which preserves the homotopy classes of $V_{2k+1,3k}$), we may write $\widehat H^1(\xi)=DH^1(\xi)$ explicitly as follows.
\beas
	&&\left(\begin{array}{c | c}
		 E_{3k\times 2} & F_{3k\times 2k-1}
	\end{array}\right) =
			\left(\begin{array}{c|c}
				\begin{array}{c}
				A_{2k\times 2} \\ 
				\hline C_{k\times 2} 
				\end{array}
			&
				\begin{array}{c}
				B_{2k\times 2k-1}\\
				\hline D_{k\times 2k-1}
				\end{array}
			\end{array}\right)\\
	&=&
		\frac{1}{4}\left(\begin{array}{c|c}
		\begin{array}{c c}
		2 & -2i\\
		0 & 0\\
		0 & 0\\
		0 & 0 \\
		\vdots & \vdots \\
		0 & 0\\
		0& 0\\
		\hline
		2z & i2z\\
		2x+t+is & 2y+i(t+is)\\
		\zt_1 &  i\zt_1\\
		\vdots & \vdots \\
		\zt_{k-2} &  i\zt_{k-2}
		\end{array} 
						& \begin{array}{c c c c c c}
							0 & 0 & 0 & \cdots & 0 & 0\\
	2\cos\theta & 0 & 0 & \cdots & 0 &-(4+2t)\sin\theta\\
							0 & 2 & 0 & \cdots & 0 & 0\\
							0 & 0 & 2 & \cdots & 0 & 0\\
							\vdots & \vdots & \vdots &  \ddots & \vdots & \vdots\\
							0 & 0 & 0 & \cdots & 2 & 0\\
	2\sin\theta & 0 &0 & \cdots & 0 & (4+2t)\cos\theta\\
				\hline
							0 & 0 & 0 &\cdots & 0 &  0\\
							z & iz & 0 &\cdots & 0 & 0\\
							0 & 0 & z & 0 & 0 & 0\\
							\vdots & \vdots & \vdots &  \ddots & \vdots & \vdots\\
							0 & 0 & 0 & \cdots & iz & 0
							\end{array}
			\end{array}\right).
\eeas
For any matrix $M_{p\times q}$ above, $\spa_\C[M]$ will denote the complex span of its $q$ columns in $\C^p$. Note that if $F_t$ is a homotopy of $F$ through $(3k)\times(2k-1)$ matrices of full rank such that $\spa_\C[F_t]\oplus\spa_\C[E]=\C^{3k}$, then $(E|F_t)$ is a homotopy of $(E|F)$ in $V_{2k+1,3k}$. With this principle in mind, we will perform the homotopy in multiple steps. First, we get rid of the factor $\frac{1}{4}$ via an elementary homotopy. 
\begin{itemize}
	\item [$(a)$] We note that the matrix $D$ has no impact on the rank of $(E|F)$. Thus, the following homotopy takes place within $V_{2k+1,3k}$.
	\beas
		(\tau,\xi)\mapsto 
		\left(\begin{array}{c|c}
				\begin{array}{c}
				A \\ 
				\hline C
				\end{array}
			&
				\begin{array}{c}
				B \\
				\hline (1-\tau) D
				\end{array}
			\end{array}\right),\quad \tau\in[0,1].
	\eeas
Thus, we assume here onwards that $D=\boldsymbol 0$. 
\item [$(b)$] With $D=\boldsymbol 0$, $\spa[F]=\spa\{z_2,...,z_{2k}\}\cong \C^{2k-1}$ is orthogonal to $\spa[E]$. Thus, if $B'\in\gl_\C(2k-1)$ denotes the matrix obtained by deleting the first row of $B$, any homotopy $B_t'$ of $B'$ in $\gl_\C(2k-1)$ induces a homotopy of $(E|F)$ in $V_{2k+1,3k}$. Now note that $B'$ only depends on $(t,\theta)\in[-1,1]\times S^1$, and $\frac{2^{-2k}}{4+2t}B'(\xi)$ is an element in $\operatorname{SU}(2k-1)$. Thus, $\xi\mapsto \frac{2^{-2k}}{4+2t} B'(\xi)$ induces an element in $\pi_1(\operatorname{SU}(2k-1))\cong 0$. So, there is a homotopy between $\xi\mapsto B'(\xi)$ and $\xi\mapsto I$ in $\gl_\C(2k-1)$. We may now assume that $B=\left(\begin{smallmatrix} \boldsymbol{0}_{1,2k-1} \\ \hline \\  \text{\bf I}_{2k-1,2k-1} \end{smallmatrix}\right)$.
\item [$(c)$] Since $\spa[E]=\spa\{z_1,z_{2k+1},...,z_{3k}\}\cong\C^k$ is orthogonal to $\spa[F]$, it suffices to produce an appropriate homotopy in $V_{2,k+1}$ of
	\beas
		E'=
				\left(\begin{array}{c c}
		2 & -2i\\
		2z & i2z\\
		2x+t+is & 2y+i(t+is)\\
		\zt_1 & i\zt_1\\
		\vdots & \vdots \\
		\zt_{k-2} & i\zt_{k-2}
		\end{array}\right),	
	\eeas
which, after an elementary homotopy in $V_{2,k+1}$, becomes 
	\beas
				\left(\begin{array}{c c}
		0 & -i\\
		z & iz\\
		t+is & i(t+is)\\
		\zt_1 & i\zt_1\\
		\vdots & \vdots \\
		\zt_{k-2} & i\zt_{k-2}
		\end{array}\right)=:(E_1|E_2).	
	\eeas
Now, the homotopy $E_\tau'=(E_1|e^{i\tau\frac{\pi}{2}} E_2+i\tau E_1)$, $\tau\in[0,1]$, gives that $\widehat H^1$ is homotopic to 
	\bes
		\xi \mapsto \wt i_{2k}(z,t,\zt,s)=
			\left(\begin{array}{c | c}	
		\mathbf{0}_{2k,1}  & \mathbf{I}_{2k,2k} \\
	\hline
  	\begin{array}{c}
		z \\ t+is\\ \zt^T
	\end{array}
					& \mathbf{0}_{k,2k}
 \end{array}\right).
	\ees
\item [$(d)$]  Now, recall that $f_1(\xi)=i_{2k}(z,y,\zt,s)$, where $i_{2k}\equiv  E\cdot \wt i_{2k}$ and $E$ is an elementary matrix in $GL_\C(3k)$ whose action on $V_{2k+1,3k}$ is to swap the lower block of $k$ rows with the upper block of $2k$ rows. As $\operatorname{GL}_\C(3k)$ is connected, we may homotope $\wt i_{2k}=E^{-1}\cdot i_{2k}$ to $i_{2k}=\text{\bf I}\cdot i_{2k}$. Hence, our claim.  
\end{itemize}
As noted earlier, a similar homotopy can be produced between $\widehat H^2$ and $f_2$. 
\end{proof}

\begin{cor}\label{cor_indexBC}
Suppose $\N$ is a tube enclosing a CR-singular curve $\gamma$ in $\C^{3k}$ such that for some $p\in\gamma$, $\N$ at $p$ is locally biholomorphic to the Beloshapka--Coffman normal form $\B^{2k+1,3k}$ at $0$ (see \eqref{eq_BCform}). Then there exists an admissible parametrization $F$ of $\N$ such that $\operatorname{ind}_F\N=1$. 
\end{cor}
\begin{proof} By the local biholomorphic invariance of the index (see Lemma~\ref{lem_index}), it suffices to show that the index of $\B^{2k+1,3k}$ at $0$ is $\pm 1$ in the following sense. For $\eps<2\pi$, if $H:\{e^{i\theta}:|\theta|<\eps\}\times D^{2k}\rightarrow \C^{3k}$ is a parametrization of some neighborhood $U\subset\B^{2k+1,3k}$ of $0$, which maps $(1,0)$ to $0$ and the curve $\{e^{i\theta}:|\theta|<\eps\}\times\{0\}$ onto $U\cap (x_{2k}$-axis$)$, then $\deg(\widehat{H}_{1\operatorname{-}\sli})=\pm 1$. Since the induced degree of any two such parametrizations will only differ by a sign (again, see Lemma~\ref{lem_index}) it suffices to produce some $H$ for which $\deg(\widehat{H}_{1\operatorname{-}\sli})=1$.  

This is essentially done in Theorem~\ref{thm_inequivmodels}, since 
	\bes
		\M_1\cap\{x_{2k}=0\}
		=\B^{2k+1,3k}\cap\left\{x_{2k}=0,
		 ||(z_1,x_2,...,x_{2k-1})||<\tfrac{1}{2}\right\}
				 \quad \text{up to translation.}
	\ees
To be more precise, we note that the map $h=(h_1,...,h_{3k})$ given by
	\bes
		h_j(z_1,...,z_{3k})=
		\begin{cases}
			z_j,\ \text{if}\ j\neq 2, 2k,\\
			\sqrt{z_2^2+z_{2k}^2}-1,\ \text{if}\ j=2,\\
			\arctan\left(\frac{x_{2k}}{x_2}\right),\ \text{if}\ j=2k,
		\end{cases}
	\ees
is a local biholomorphism near $p=(0,1,0,...,0)\in \M_1$ which maps a neighborhood of $p$ in $\M_1$ to a neighborhood of $0$ in $\B^{2k+1,2k}$. The local biholomorphic invariance of the index implies that $H=h\circ H^1$, where $H^1$ is as in Theorem~\ref{thm_inequivmodels}, is the desired parametrization of $\B^{2k+1,3k}$ near $0$. 
\end{proof}

We now indicate how $\M_1$ and $\M_2$ can be modified to produce polynomially convex models of all indices. For any nonnegative index $n$, one can set $z_{2k+1}$ as $\zbar_1^{n+1}$ and $\zbar_1^{n+1}(x_2+ix_{2k})$ in the definitions of $\M^1$ and $\M^2$, respectively, and modify $H^1$ and $H^2$ in the obvious way. As in Theorem~\ref{thm_inequivmodels}, these two models will represent the same and only TRC cobordism class of index $n$ when $k$ is even, and the only two distinct TRC cobordism classes of index $n$ when $k$ is odd. For tubes of negative indices, we note that if $(\N,F)$ is a parametrized tube of index $n$, $n\in\mathbb N$, then $(\N,F^*)$ is a parametrized tube of index $-n$, where $F^*:(\theta,\eta)\mapsto(\theta,\eta^*)$, $(\theta,\eta)\in S^1\times D^{2k}$ (see Remark~\ref{rem_index}). 

\subsection{A fully-TR tube containing an Alexander set}\label{subsec_Alexander} At the end of Subsection~\ref{subsec_models}, we construct polynomially convex parametrized tubes of index zero. These are in fact fully-TR tubes. In the proof of Theorem~\ref{thm_anstructure}, we will use a fully-TR tube with an extremely different complex-analytic behavior from the ones mentioned above. This construction is based on the following result due to Alexander (\cite{Al98}) which has been heavily used in the context of hulls without analytic structure (see \cite{IzSt18} and \cite{ArWo17}). {\em The standard torus $\mathbb{T}^2=\{(e^{i\theta},e^{i\psi}):\theta,\psi\in\rl\}$ in $\CC$ contains a compact subset $E$ such that $\widehat E\setminus E$ is nonempty but contains no analytic disk. Such a set can be chosen in any neighbourhood of the diagonal of $\mathbb{T}^2$.}

Now, for some $\eps<<1$, let $F=(f_1,...,f_{3k}):S^1\times \eps D^{2k}\rightarrow \C^{3k}$ be the smooth map given by
	\beas
		&(f_1,...,f_{2k+1}):(\theta,x,y,t,u_1,v_1,...,u_{2k},v_{2k},s)
			\mapsto\left(e^{i\theta},e^{i(\theta+s)},x, y,t,u_1,v_1,...,u_{2k},v_{2k}\right),&\\
		&f_{j}\equiv 0, 2k+2\leq j\leq 3k.& 
	\eeas

\begin{lemma}\label{lem_Alexander}
Let $\A:=F(S^1\times \eps D^{2k})$ and $G:(\theta,\eta)\mapsto F(\theta,\eps\eta)$, $(\theta,\eta)\in S^1\times D^{2k}$. Then, $\A$ is a fully-TR tube in $\C^{3k}$
and $[\widehat G]$ is trivial in $[S^1\times S^{2k-1},W_{2k+1,3k}]$. Moreover, $\A$ contains a compact set $E$ such that such that $\widehat E\setminus E$ is nonempty but contains no analytic disk. 
\end{lemma}
\begin{proof} The first part of the claim follows from computing $DG$ and observing that it depends only on $(\theta,s)\in S^1\times [-1,1]$. Since $\pi_1(W_{2k+1,3k})$ is trivial, so is $[\widehat G]$.  

For the second part of the claim, note that $\A$ contains a neighborhood $V$ of the diagonal of the standard torus in the $z_1$-$z_2$ plane. In fact, $V=F(\{x=y=t=u_j=v_j=0\})$. We note that if $E\subset V$ is the set granted by Alexander's theorem above, then $E\times\{0\}^{3k-2}$ is a compact set in $\A$ with same hulls as $E$. Thus, in an abuse of notation, denoting $E\times\{0\}^{3k-2}$ by $E$ completes the proof. 
\end{proof}

\section{Proofs of the main results}\label{sec_proofs} We first isolate a lemma that is used in both the proofs in this section. The proof uses standard arguments; see \cite[Lemma~2.1]{GuSh20} for a version of this lemma. Here, for $p\in\Cn$ and $r>0$, $B_p(r)=\{z\in\Cn:||z-p||<r\}$ and $\overline B_p(r)$ denotes its closure in $\Cn$ 

\begin{lemma}\label{lem_unions} Suppose $X\subset\Cn$ is a polynomially convex compact set, and $p_1,...,p_\ell\in\Cn\setminus X$. Then, there exist $r_1,...,r_\ell>0$, so that $X\cup\bigcup_{j=1}^\ell Q_j$ is polynomially convex, for any choice of polynomially convex compact sets $Q_j\subseteq\bar B_{p_j}(r_j)$, $j=1,...,\ell$.
\end{lemma}

\subsection{Proof of Theorem~\ref{t.main}}\label{subsec_mainproof} Let $M$ be as given. In view of Lemma~\ref{lem_emb2} and Corollary~\ref{cor_indexBC}, we may assume that $M$ is embedded in $\C^{3k}$, its set of CR-singularities is a union of disjoint simple closed curves, $C_1$,..., $C_\ell$, and each $C_j$ is enclosed in a parametrized tube $(\N_j,F_j)$, where $\N_j\subset M$ and $\operatorname{ind}_{F_j}(\N_j)=1$. Using the results of Section~\ref{sec_tubes}, we now proceed to ``replace" the $\N_j$'s with polynomially convex tubes enclosing CR-singular curves. 

\begin{lemma}\label{lem_arrange}
There	 exist $\ell$ disjoint tubes, $\K_1$,...,$\K_\ell$,  enclosing CR-singular curves $\gamma_1$,...,$\gamma_\ell$ such that 
	\begin{itemize}
\item [$(i)$] each $\K_j$ is disjoint from $M$,
\item [$(ii)$]  for some parametrization $G_j$, $(\K_j,G_j)$ is TRC cobordant to $(\N_j,F_j)$, $j=1,...,\ell$,
\item [$(iii)$] $\Gamma=\gamma_1\cup\cdots \cup\gamma_\ell$ is polynomially convex in $\C^{3k}$, and 
\item [$(iv)$] $\K=\K_1\cup\cdots \cup \K_\ell$ is polynomially convex in $\C^{3k}$.  
\end{itemize}
\end{lemma}
\begin{proof} First, since $\operatorname{ind}_{F_j}(\N_j)=1$, $(\N_j,F_j)$ is TRC cobordant to either $(\M^1,H^1)$ or $(\M^2,H^2)$ (which are TRC cobordant to each other when $k$ is even). This follows from the characterization of TRC cobordism classes obtained in Lemma~\ref{lem_TRcob}. Without loss of generality, we assume that there is some $\ell_0\in\{1,...,\ell\}$ such that $(\N_1,F_1),...,(\N_{\ell_0},F_{\ell_0})$ are TRC cobordant to $(\M^1,H^1)$, and $(\N_{\ell_0+1},F_{\ell_0+1}),...,(\N_\ell,F_\ell)$ are TRC cobordant to $(\M^2,H^2)$.

Next, let $p_1,...,p_\ell\in\C^{3k}\setminus M$ be distinct points. Let  $r_1,...,r_\ell>0$ be as granted by Lemma~\ref{lem_unions} (for $X=\emptyset$). By shrinking the $r_j$'s further, we may assume that each $\overline B_{p_j}(r_j)$ is disjoint from $M$. Now observe that the CR structure, polynomial convexity (both of the tubes and the enclosed CR-singular curves) and the TRC cobordism class of $(\M^1,H^1)$ and $(\M^2,H^2)$ remain unchanged under transformations of the form $A:z\mapsto rz-p$, where $r>0$ and $p\in\C^{3k}$. For such a map $A$, we call $(A(\M^j),H^j\circ A)$ a {\em copy} of $(\M^j,H^j)$, and it is a parametrized tube in $B_p(r)$ enclosing the CR-singular curve $A(\gamma)$. Now, let $(K_1,G_1),...,(K_{\ell_0},G_{\ell_0})$ be copies of $(\M_1,H^1)$ in $B_{p_1}(r_1),...,B_{p_{\ell_0}}(r_{\ell_0})$, respectively, and $(K_{\ell_0+1},G_{\ell_0+1}),...,(K_{\ell},G_{\ell})$ be copies of $(\M_2,H^2)$ in $B_{p_{\ell_0+1}}(r_{\ell_0+1}),...,B_{p_{\ell}}(r_{\ell})$, respectively. Denote by $\gamma_j$ the CR-singular curve enclosed by $K_j$, $j=1,...,\ell$.

Finally, note that $(i)$ and $(ii)$ hold by construction. Since $K_1,...,K_\ell$ and $\gamma_1,...,\gamma_\ell$ are polynomially convex subsets of $\C^{3k}$ contained in the balls $B_{p_1}(r_1),...,B_{p_\ell}(r_\ell)$, respectively, Lemma~\ref{lem_unions} gives $(iii)$ and $(iv)$. This concludes the proof of Lemma~\ref{lem_arrange}.
\end{proof}

 Fix a $j\in\{1,...,\ell\}$. By the neighborhood replacement result, i.e., Proposition~\ref{prop_NRR}, there is a smooth embedding $\Psi_{\N_j,K_j}:S^1\times 2D^{2k}\rightarrow\C^{3k}$ such that if we set $\mathcal T_j:=\Psi_{\N_j,K_j}(S^1\times 2D^{2k})$, then
	\begin{enumerate}
\item [$(a)$] $\bdy \mathcal T_j=\bdy \N_j$;
\item [$(b)$] a collar of $\bdy \mathcal T_j$ in $\mathcal T_j$, say $\operatorname{coll}(\bdy \mathcal T_j)$, coincides with a collar of $\bdy \N_j$ in $\N_j$;
\item [$(c)$] $K_j\subset \mathcal T_j$, and $\mathcal T_j$ is totally real everywhere except along $\gamma_j\subset K_j$.  
\end{enumerate} 
We let $\N_j'\subset \N_j$ be a slightly thinner tubular neighborhood of $C_j$ chosen so that $\bdy \N_j'\subset \operatorname{coll}(\bdy \mathcal T_j)$. Now, letting $M^*=M\setminus\left(\N_1'\cup\cdots\cup \N_\ell'\right)$, we set   
	\bes
		M_1=M^*\cup\left(\bigcup_{j=1}^\ell \mathcal T_j\right).
	\ees
We note the inclusion map $\iota:M_1\rightarrow\C^{3k}$ is an immersion of $M$ whose image is totally real except along $\Gamma=\gamma_1\cup\cdots\cup\gamma_\ell$, which admits a polynomially convex neighborhood $K$ in $M_1$. Moreover, $\iota$ is a smooth embedding when restricted to a sufficiently small open neighborhood of $K$ in $M_1$. Thus, by the relative version of the weak Whitney embedding theorem, we may perturb $\iota$, keeping it fixed over $K$, to obtain a smooth embedding $j:M\rightarrow \C^{3k}$ such that $K\subset j(M)$, and $j(M)\setminus \Gamma$ is totally real. 

Note that $j(M)\setminus K^\circ $ is a compact totally real smooth submanifold of $\C^{3k}$ with boundary, and $K$ is polynomially convex. Thus, we can apply the following result due to Arosio--Wold (see~\cite[Theorem~1.4]{ArWo17}). {\em Let $N$ be a compact smooth manifold (possibly with boundary) of dimension $d<n$ and let $f:N\rightarrow\Cn$ be a totally real $\cont^\infty$-embedding. Let $X\subset\Cn$ be a compact polynomially convex set. Then for all $s\geq 1$ and for all $\eps>0$, there exists a totally real $\cont^\infty$-embedding $f_\eps:N\rightarrow\Cn$ such that
	\begin{enumerate}
		\item [$a.$] $||f-f_\eps||_{\cont^s(N)}<\eps$ 
		\item [$b.$] $f_\eps=f$ on $f^{-1}(X)$, and
		\item [$c.$] $\widehat{X\cup f_\eps(N)}=X\cup f_\eps (N)$. 
	 \end{enumerate}} 
We set $N=M\setminus j^{-1}(K^\circ)$, $X=K$ and $f=j|_N$, and let $\eps>0$ be arbitrary but fixed. Then, $M'=f_\eps(N\setminus j^{-1}(K))\cup K$ is an embedded submanifold of $\C^{3k}$ that is diffeomorphic to $M$, polynomially convex, and totally real away from $\Gamma\subset M'$, which is a union of finitely many disjoint simple closed curves. This gives items $(1)$ and $(2)$ of the theorem.  

It remains to show that $\cont(M')=\mathcal P(M')$. For this, we use the following result due to O'Farrel--Preskenis--Walsch; see~\cite{OFPrWa84} or \cite[\S 6.3]{St07}. {\em Let $X$ be a compact holomorphically convex set in $\Cn$, and let $X_0$ be a closed subset of $X$ for which $X\setminus X_0$ is a totally real subset of the manifold $\Cn\setminus X_0$. A function $f\in\cont(X)$ can be approximated uniformly on $X$ by functions holomorphic on a neighbourhood of $X$ if and only if $f|_{X_0}$ can be approximated uniformly on $X_0$ by functions holomorphic on $X$.} 

First, we apply the above result (or an earlier version by Harvey--Wells; see [HaWe71]) to $X=\Gamma$ and $X_0=\emptyset$ to obtain that any $f\in\cont(\Gamma)$ can be approximated uniformly on $\Gamma$ by functions holomorphic on a neighbourhood of $\Gamma$. Since $\Gamma$ is polynomially convex, the Oka--Weil theorem allows us to conclude that $\cont(\Gamma)=\mathcal P(\Gamma)$. Thus, $X:=M'$ and $X_0=\Gamma$ satisfy the hypothesis of the O'Farrel--Preskenis--Walsch result, and any $f\in\cont(M')$ can be approximated uniformly on $M'$ by functions holomorphic on a neighbourhood of $M'$. Once again, applying the Oka--Weil theorem to the polynomially convex set $M'$, we obtain that $\mathcal \cont(M')=\mathcal P(M')$.

\subsection{Proof of Theorem~\ref{thm_anstructure}} We let $\iota:M\rightarrow\C^{3k}$ be the embedding granted by Theorem~\ref{t.main}. Recall from the proof of the theorem that the CR-singular set $\Gamma$ of $\iota(M)$ is contained in a polynomially convex compact set $K\subset \iota(M)$. Now, by Lemma~\ref{lem_unions}, we can find a small closed ball $Y\subset\C^{3k}$ that is disjoint from $\iota(M)$ and is such that $K\cup Q$ is polynomially convex for any polynomially convex subset $Q\subset Y$. We now find a copy of $\A$ --- the tube constructed in Subsection~\ref{subsec_Alexander} --- in the ball $Y$, by using an appropriate map of the form $z\mapsto rz-p$, $r>0$, $p\in\C^{3k}$. We abuse notation and call this copy and the Alexander set contained in it $\A$ and $E$, respectively. Since $\widehat E\subset Y$, $K\cup \widehat E$ is polynomially convex.   

Now, let $\N$ be a fully-TR tube in some small totally real ball $B\subset\iota(M)$. Then, by taking any parametrization of $B$ by $D^{2k+1}$ and restricting it to a copy of $S^1\times D^{2k}$ in $D^{2k+1}$, we obtain an admissible parametrization $\N$ whose frame map induces the  trivial element in $[S^1\times S^{2k-1},W_{2k+1,3k}]$. Using the neighborhood replacement result and Lemma~\ref{lem_Alexander}, we replace $\N$ by $\A$. Combined with the relative version of the weak Whitney Embedding Theorem, this gives an embedding $j:M\rightarrow\C^{3k}$ such that 
	\begin{itemize}
\item [$*$] $j$ coincides with $\iota$ on $\iota^{-1}(K)$;
\item [$*$] $j(M)\setminus\Gamma$ is totally real; and 
\item [$*$] there is a totally real ball $B'\subset j(M)$ that contains $\A$.
	\end{itemize}  

We now apply the Aroiso--Wold perturbation result stated in Subsection~\ref{subsec_mainproof} to $N=M\setminus j^{-1}(K)$, $X=\widehat E\cup K$ and $j|_N$. We obtain a smooth embedding of $M$ into $\C^{3k-1}$ whose image $\Sigma$ contains $E\cup K$, is totally real away from $\Gamma$, and $\widehat{\Sigma\cup \widehat E}= \Sigma\cup\widehat E$. Thus, $\widehat{\Sigma}=\widehat{\Sigma\cup E}= \Sigma\cup\widehat E$. Now, if $\widehat E$ was contained in $\Sigma$, then $\Sigma$ would be a polynomially convex manifold that is totally real except along the polynomially convex set $\Gamma$. This is precisely the situation handled in the proof of Theorem~\ref{t.main} using the O'Farrel--Preskenis--Walsch result cited above. Using the same technique, we see that any subset $T$ of such a manifold has the property that $\cont(T)=\mathcal{P}(T)$, and thus is polynomially convex. This contradicts the fact that $ E\subset \Sigma$ is not polynomially convex. Thus, 
$\widehat\Sigma\setminus\Sigma$ is nonempty but contains no analytic disk. 

To show that $\widehat{\Sigma}=h_r(\Sigma)$, we use an argument due to Izzo, who showed in \cite[Section 3]{Iz19} that $E$ satisfies the generalized argument principle, i.e., if $p$ is a polynomial that has a continuous logarithm on $E$, then $0\notin p(E)$. We now apply the following result due to Stolzenberg (\cite{St63}) to $X=E$ and $Y=B'$. {\em If $X\subseteq Y\subset\Cn$ are compact sets such that $X$ satisfies the generalized argument principle and the first {\v C}ech cohomology group $\check{\mathrm{H}}^1(Y,\mathbb{Z})$ vanishes, then $\widehat X\subset h_r(Y)$.} Thus, $\widehat{\Sigma}= \Sigma\cup\widehat{E}\subseteq h_r(\Sigma)$, and the two hulls coincide. 

\bibliography{PolyCvxBiblio}
\bibliographystyle{plain}
\end{document}